\documentclass[reqno]{amsart}
\usepackage{amsmath,amsfonts,amssymb}
\newtheorem{teor}{Theorem}[section]

\newtheorem{prop}[teor]{Proposition}

\theoremstyle{definition}
\newtheorem{defi}[teor]{Definition}

\numberwithin{equation}{section}
\newcommand{\R}{\mathbb{R}}

\newcommand{\W}{\Omega}
\newcommand{\Om}{\Omega}
\newcommand{\w}{\omega}

\newcommand{\di}{\textsf{d}}

\newcommand{\wit}{\widetilde}
\newcommand{\des}{\displaystyle}

\newcommand{\n}[1]{\| #1 \|}

\DeclareMathOperator{\Int}{Int} 
\DeclareMathOperator{\spa}{span} \DeclareMathOperator{\rg}{Rg}
\begin{document}
\title[Persistence in  quasimonotone  parabolic PFDEs with delay]{Persistence in non-autonomous quasimonotone parabolic partial functional differential equations with delay}
%\author[S. Novo]{Sylvia Novo}
%\author[C. N\'{u}\~{n}ez]{Carmen N\'{u}\~{n}ez}
\author[R. Obaya]{Rafael Obaya}
\author[A.M. Sanz]{Ana M. Sanz}
\address[R. Obaya]{Departamento de Matem\'{a}tica
Aplicada, E. Ingenier\'{\i}as Industriales, Universidad de Valladolid,
47011 Valladolid, Spain.} %\email{sylnov@wmatem.eis.uva.es}
\email{rafoba@wmatem.eis.uva.es}
%\address[C. N\'{u}\~{n}ez and R. Obaya]{Departamento de Matem\'{a}tica
%Aplicada, E. Ingenier\'{\i}as Industriales, Universidad de Valladolid,
%47011 Valladolid, Spain, and member of IMUVA, Instituto de
%Matem\'{a}ticas, Universidad de Valladolid.}
% \email{carnun@wmatem.eis.uva.es} \email{rafoba@wmatem.eis.uva.es}
\address[A.M. Sanz]{Departamento de Did\'{a}ctica de las Ciencias Experimentales, Sociales y de la Matem\'{a}tica,
E.U. Educaci\'{o}n, Universidad de Valladolid, 34004 Palencia, Spain.} \email{anasan@wmatem.eis.uva.es}
\address[A.M. Sanz and R. Obaya]{IMUVA, Instituto de Investigaci\'{o}n en
Matem\'{a}ticas, Universidad de Valladolid, Spain.}
\thanks{The authors were partly supported by MINECO/FEDER  grant
MTM2015-66330-P, and the European Commission under project H2020-MSCA-ITN-2014 643073 CRITICS}
\date{}
\dedicatory{Dedicated to Peter E. Kloeden on the occasion of his 70th birthday}
\begin{abstract}
This paper provides a dynamical frame to study non-autonomous  parabolic
partial differential equations with finite delay. Assuming monotonicity of the linearized semiflow, conditions for the existence of a continuous separation of
type~II over a minimal set are given. Then, practical criteria for the
uniform or strict persistence of the  systems above a minimal set are
obtained.
\end{abstract}
\keywords{Topological dynamics, skew-product semiflows, continuous separation, partial functional differential
equations with finite delay, uniform and strict persistence}
\subjclass{35K58, 37B55, 37C60, 37C65}
\renewcommand{\subjclassname}{\textup{2010} Mathematics Subject Classification}
\maketitle
\section{Introduction}\label{intro}\noindent
%%%%%%%%%%%%%%%%%%%%%%%%%%%%%%%%%%%%%%%%%%%%%%%%%%%%%%%%%%%%%%%%%%%%%%%%%%%%%%%%%%%%%%%%%%%%%%%5
In this paper we investigate some qualitative properties of the  skew-product semiflows generated by the solutions of non-autonomous parabolic partial functional differential equations (PFDEs for short) with finite delay and boundary conditions of Neumann, Robin or Dirichlet type. In this non-autonomous framework the phase space is a product space $\Omega \times C$, where the base $\Omega$ is a compact metric space under the action of a continuous flow $\sigma: \R \times \Omega \to \Omega$, $(t,\omega) \mapsto \omega {\cdot} t$ and the state space $C$  is an infinite dimensional Banach space of continuous functions. The skew-product semiflow $\tau: \R_+\times \Omega \times C \to \Omega \times C$, $(t,\omega,\varphi) \mapsto (\omega{\cdot} t,u(t,\omega, \varphi))$ is built upon the mild solutions of the associated abstract Cauchy problems (ACPs for short) with delay. We assume that the flow in the base is minimal.
\par
This formalism permits to carry out a dynamical study of the solutions of non-autonomous differential equations in which the temporal variation of the coefficients is almost periodic, almost automorphic or, more generally, recurrent. Frequently, $\Omega$ is obtained as the hull of the non-autonomous function defining the differential equations, although the approach considered here is more general. The references Ellis~\cite{elli}, Johnson et al.~\cite{jonnf}, Sacker and Sell~\cite{sase, sase94}, and Shen and Yi~\cite{shyi}, and references therein, contain ingredients of the theory of non-autonomous dynamical systems which will be used throughout this work.
\par
The main issue in the paper is the persistence of the systems of parabolic PFDEs.
Persistence is a dynamical property which has a great interest in mathematical modelling, in areas such as biological population dynamics, epidemiology, ecology or neural networks. In the field of monotone dynamical systems, different notions of persistence have been introduced, with the general meaning that in the long run the trajectories place themselves above a prescribed region of the phase space,  which we take to be a minimal set $K\subset \W\times C$. In many applications this minimal set  is  $\Omega \times \{0\}$, so that, roughly speaking, uniform or strict persistence  means that the solutions eventually become uniformly strongly or strictly positive, respectively.
\par
In~\cite{obsa18} Obaya and Sanz showed that in the general non-autonomous setting,  in order that persistence can be detected experimentally, this notion should  be considered  as a collective property of the complete family of systems over $\Om$. We follow this collective approach  to develop dynamical properties of persistence with important practical implications. Our study intends to extend the theory of persistence written in Novo et al.~\cite{noos7} and Obaya and Sanz~\cite{obsa} for non-autonomous  ODEs, FDEs with delay and parabolic PDEs to parabolic PFDEs, considering also the cases of Robin or Dirichlet boundary conditions.
\par
We briefly explain the structure and contents of the paper. Some basic concepts in the theory of non-autonomous dynamical systems are included in Section~\ref{sec-prelim}. Section~\ref{sec-skew-product} is devoted to describe the dynamical scenario in which the parabolic problems are immersed, distinguishing the case of Neumann or Robin boundary conditions, and the Dirichlet case. We analyze the regularity properties and the long-term behaviour of the solutions that determine the topological structure of omega-limit sets and minimal sets. We follow arguments in the line of Martin and Smith~\cite{masm0,masm} and  Wu~\cite{wu} to extend previous results given in Novo et al.~\cite{nonuobsa} for Neumann boundary conditions, to the case of Robin and Dirichlet boundary conditions. We also study the consequences of the so-called quasimonotone condition in the problems.
\par
In Section~\ref{sec-linearized sem}, under regularity conditions in the reaction terms in the equations, we build the variational problems along the semiorbits of $\tau$, whose mild solutions induce the linearized skew-product semiflow. Then, in the linear and monotone setting, we consider continuous separations of type~II and the associated principal spectrums, which can be determined by Lyapunov exponents. The classical concept of continuous separation given by Pol\'{a}\v{c}ik and Tere\v{s}\v{c}\'{a}k~\cite{pote} in a discrete dynamical setting, and then extended to a continuous setting by Shen and Yi~\cite{shyi}, has proved to be widely applicable in non-autonomous ODEs and parabolic PDEs, but not in equations with delay. Later, Novo et al.~\cite{noos6} introduced a variation of this notion, and called it continuous separation of type~II, in order to make it applicable to delay equations. The results in Novo et al.~\cite{noos7} and Obaya and Sanz~\cite{obsa,obsa18} show its importance in the dynamical description of non-autonomous FDEs with finite delay, and now it becomes relevant in reaction-diffusion systems with delay.
\par
Finally, in Section~\ref{sec-uniform persistence} we consider regular and quasimonotone parabolic PFDEs.  %determine conditions in the reaction term so as to have a monotone and $C^1$ skew-product semiflow, as well as some regularity conditions in the solutions in order to apply the parabolic maximum principles.
Assuming the existence of a minimal set $K$ for $\tau$ with a flow extension, %we use arguments of ergodic theory and topological dynamics to investigate the presence of uniform or strict persistence above  the minimal set $K$.
we first establish an easy criterion for the existence of a continuous separation of type~II over $K$ in terms of the irreducibility of a constant matrix calculated from the partial derivatives of the reaction term in the equations, with respect to the non-delayed and delayed state components.  A key fact is that in the general case, after a convenient permutation of the variables in the system, the constant matrix mentioned before has a block lower triangular structure, with irreducible diagonal blocks. This permits to consider a family of lower dimensional linear systems with a continuous separation, for which the property of persistence depends upon the positivity of its principal spectrum. In this situation, a sufficient condition for the presence of uniform or strict  persistence in the area  above $K$ is given in terms of the principal spectrums of an adequate subset of such systems in each case.
%%%%%%%%%%%%%%%%%%%%%%%%%%%%%%%%%%%%%%%%%%%%%%%%%%%%%%%%%%%%%%%%%%%%%%%%%%%%%%%%5
%%%%%%%%%%%%%%%%%%%%%%%%%%%%%%%%%%%%%%%%%%%%%%%%%%%%%%%%%%%%%%%%%%%%%%%%%%%%%%%%%%55
\section{Some preliminaries}\label{sec-prelim}\noindent
In this section we include some basic notions in
topological dynamics for non-autonomous dynamical systems.
\par
Let $(\W,d)$ be a compact
metric space. A real {\em continuous flow \/} $(\W,\sigma,\R)$ is
defined by a continuous map $\sigma: \R\times \W \to  \W,\;
(t,\w)\mapsto \sigma(t,\w)$ satisfying  $\sigma_0=\text{Id}$,  and  $\sigma_{t+s}=\sigma_t\circ\sigma_s$ for each $t, s\in\R$,
where $\sigma_t(\w)=\sigma(t,\w)$ for all $\w \in \W$ and $t\in \R$.
The set $\{ \sigma_t(\w) \mid t\in\R\}$ is called the {\em orbit\/}
of the point $\w$. A subset $\W_1\subset \W$ is {\em
$\sigma$-invariant\/} if $\sigma_t(\W_1)=\W_1$ for every $t\in\R$, and it is {\em minimal \/} if it is compact,
$\sigma$-invariant and it does not contain properly any other
compact $\sigma$-invariant set. Every compact
and $\sigma$-invariant set contains a minimal subset. Furthermore, a
compact $\sigma$-invariant subset is minimal if and only if every
orbit is dense. We say that the continuous flow $(\W,\sigma,\R)$ is
{\em recurrent\/} or {\em minimal\/} if $\W$ is minimal.
\par
A finite regular measure defined on the Borel sets of $\W$ is called
a Borel measure on $\W$. Given $\mu$ a normalized Borel measure on
$\W$, it is {\em $\sigma$-invariant\/} if $\mu(\sigma_t(\W_1))=\mu(\W_1)$ for every Borel subset
$\W_1\subset \W$ and every $t\in \R$. It is {\em ergodic\/}  if, in
addition, $\mu(\W_1)=0$ or $\mu(\W_1)=1$ for every
$\sigma$-invariant Borel subset $\W_1\subset \W$.
\par
Let $\R_+=\{t\in\R\,|\,t\geq 0\}$. Given a continuous compact flow $(\W,\sigma,\R)$ and a
complete metric space $(C,\di)$, a continuous {\em skew-product semiflow\/} $(\W\times
C,\tau,\,\R_+)$ on the product space $\W\times C$ is determined by a continuous map
\begin{equation*}
 \begin{array}{cccl}
 \tau \colon  &\R_+\times\W\times C& \longrightarrow & \W\times C \\
& (t,\w,\varphi) & \mapsto &(\w{\cdot}t,u(t,\w,\varphi))
\end{array}
\end{equation*}
 which preserves the flow on $\W$, denoted by $\w{\cdot}t=\sigma(t,\w)$ and referred to as the {\em base flow\/}.
 The semiflow property means that $\tau_0=\text{Id}$, and $\tau_{t+s}=\tau_t \circ \tau_s$ for all  $t, s\geq 0$,
where again $\tau_t(\w,\varphi)=\tau(t,\w,\varphi)$ for each $(\w,\varphi) \in \W\times C$ and $t\in \R_+$.
This leads to the so-called semicocycle property:
\begin{equation}\label{semicocycle}
 u(t+s,\w,\varphi)=u(t,\w{\cdot}s,u(s,\w,\varphi))\quad\text{for }\; t,s\ge 0\;\;\text{and }\; (\w,\varphi)\in \W\times C\,.
\end{equation}
\par
The set $\{ \tau(t,\w,\varphi)\mid t\geq 0\}$ is the {\em semiorbit\/} of
the point $(\w,\varphi)$. A subset  $K$ of $\W\times C$ is {\em positively
invariant\/} if $\tau_t(K)\subseteq
K$ for all $t\geq 0$ and it is $\tau$-{\em invariant\/} if $\tau_t(K)=
K$ for all $t\geq 0$.  A compact positively invariant set $K$ for the
semiflow  is {\em minimal\/} if it does not contain any nonempty
compact positively invariant set  other than itself. The restricted
semiflow on a compact and $\tau$-invariant set $K$ admits a {\em
flow extension\/} if there exists a continuous flow
$(K,\wit\tau,\R)$ such that $\wit \tau(t,\w,\varphi)=\tau(t,\w,\varphi)$ for all
$(\w,\varphi)\in K$ and $t\in\R_+$.
\par
Whenever a semiorbit $\{\tau(t,\w_0,\varphi_0)\mid t\ge 0\}$ is relatively
compact, one can consider the {\em omega-limit set\/} of
$(\w_0,\varphi_0)$, formed by the
limit points of the semiorbit as $t\to\infty$. The omega-limit set is then a
nonempty compact connected and $\tau$-invariant set. An important
property of an omega-limit set  is that semiorbits admit backward
extensions inside it. Therefore, the sufficient condition for such a
set to have a flow extension is the uniqueness of backward orbits
(see~\cite{shyi} for more details).
\par
In this paper we will sometimes work under some differentiability
assumptions. More precisely, when $C$ is a Banach space, the
skew-product semiflow $\tau$ is said to be of class $C^1$
when $u$ is assumed to be of class $C^1$ in $\varphi$, meaning
that $D_{\!\varphi}u(t,\w,\varphi)$ exists for any $t>0$ and any $(\w,\varphi)\in\W\times
C$ and for each fixed $t>0$, the map $(\w,\varphi)\mapsto D_{\!\varphi}u(t,\w,\varphi)\in
\mathcal L(C)$ is continuous in a neighborhood of any compact set
$K\subset \W\times C$, for the norm topology on $\mathcal L(C)$; moreover, for any $\phi\in C$, $\lim_{\,t\to
0^+}D_{\!\varphi}u(t,\w,\varphi)\,\phi=\phi $ uniformly for $(\w,\varphi)$ in compact sets of
$\W\times C$.
\par
In that case, whenever $K\subset \W\times C$ is a compact positively
invariant set, we can define a continuous  linear skew-product
semiflow called the {\em linearized skew-product semiflow\/} of
$\tau$ over $K$,
\begin{equation*}
 \begin{array}{cccl}
 L: & \R_+\times K \times C& \longrightarrow & K \times C\\
& (t,(\w,\varphi),\phi) & \mapsto &(\tau(t,\w,\varphi),D_{\!\varphi}u(t,\w,\varphi)\,\phi)\,.
\end{array}
\end{equation*}
We note that $D_{\!\varphi}u$ satisfies the linear semicocycle property:
\begin{equation}\label{linear semicocycle}
D_{\!\varphi}u(t+s,\w,\varphi)=D_{\!\varphi}u(t,\tau(s,\w,\varphi))\,D_{\!\varphi}u(s,\w,\varphi)\quad\text{for }\; t,s\ge 0\,,\; (\w,\varphi)\in K.
\end{equation}
\par
Finally, we include the definition of monotone skew-product semiflow. A
Banach space $X$ is  {\em ordered} if there is a closed
convex cone, i.e., a nonempty closed subset $X_+\subset X$
satisfying $X_+\!+X_+\subset X_+$, $\R_+ X_+\!\subset X_+$ and
$X_+\cap(-X_+)=\{0\}$. If besides the positive cone has a nonempty interior, $\Int
X_+\not=\emptyset$, $X$ is {\em strongly ordered}. The (partial) {\em strong order relation\/} in
$X$ is then defined by
\begin{equation}\label{order}
\begin{split}
 v_1\le v_2 \quad &\Longleftrightarrow \quad v_2-v_1\in X_+\,;\\
 v_1< v_2  \quad &\Longleftrightarrow \quad v_2-v_1\in X_+\;\text{ and }\;v_1\ne v_2\,;
\\  v_1\ll v_2 \quad &\Longleftrightarrow \quad v_2-v_1\in \Int X_+\,.\qquad\quad\quad~
\end{split}
\end{equation}
The relations $\ge,\,>$ and $\gg$ are defined in the obvious way.
If $C$ is an ordered Banach space, the skew-product semiflow $(\W\times C,\tau,\R_+)$ is {\em
monotone\/} if
\begin{equation*}
u(t,\w,\varphi)\le u(t,\w,\phi)\,\quad \text{for\, $t\ge 0$, $\w\in\W$
\,and\, $\varphi, \phi\in C$ \,with\, $\varphi\le \phi$}\,.
\end{equation*}
Note that monotone semiflows are forward dynamical systems on
ordered Banach spaces which preserve the order of initial states
along the semiorbits.
%%%%%%%%%%%%%%%%%%%%%%%%%%%%%%%%%%%%%%%%%%%%%%%%%%%%%%%%%%%%%%%%%%%%%%%%%%%%
%%%%%%%%%%%%%%%%%%%%%%%%%%%%%%%%%%%%%%%%%%%%%%%%%%%%%%%%%%%%%%%%%%%%%%%%%%%
\section{Skew-product semiflows induced by  parabolic PFDEs with
delay}\label{sec-skew-product}\noindent
In this section we consider time-dependent families of initial boundary value (IBV for short) problems given by systems of parabolic PFDEs with a fixed
delay (just taken to be $1$) over a minimal flow $(\W,\sigma,\R)$, with Dirichlet, Neumann or Robin boundary conditions.  More precisely, for each $\w\in\W$ we consider the IBV problem
\begin{equation*}
\left\{\begin{array}{l} \des\frac{\partial y_i}{\partial t}(t,x)=
 d_i\Delta y_i(t,x)+f_i(\w{\cdot}t,x,y(t,x),y(t-1,x))\,,\quad t>0\,,\;x\in \bar U,\\[.2cm]
\alpha_i(x)\,y_i(t,x)+\delta_i\,\des\frac{\partial y_i}{\partial n}(t,x) =0\,,\quad  t>0\,,\;\,x\in \partial U,\\[.2cm]
 y_i(s,x)=\varphi_i(s,x)\,,\quad  s\in [-1,0]\,,\;\,x\in \bar U,
\end{array}\right.
\end{equation*}
for $i=1,\ldots,n$, where $\w{\cdot}t$ denotes the flow on $\W$;  $U$, the spatial domain, is a bounded, open and
connected  subset of $\R^m$ ($m\geq 1$) with a sufficiently smooth boundary
$\partial U$; $\Delta$ is the Laplacian operator on $\R^m$ and $d_1,\ldots,d_n$ are positive constants called the diffusion coefficients;  the map $f:\W\times \bar U\times\R^n\times\R^n\to \R^n$, called the reaction term, with components $f=(f_1,\ldots,f_n)$
satisfies the following condition:
\begin{itemize}
\item[(C)] $f(\w,x,y,\wit y)$ is continuous, and it is Lipschitz in $(y,\wit y)$ in bounded sets, uniformly for $\w\in \W$ and $x\in\bar U$, that is, given any $\rho>0$ there exists an $L_\rho>0$ such that
\[
\|f(\w,x,y_2,\wit y_2)-f(\w,x,y_1,\wit y_1)\|\leq L_\rho\,(\|y_2-y_1\|+\|\wit y_2-\wit y_1\|)
\]
for any $\w\in \W$, $x\in\bar U$ and $y_i,\,\wit y_i\in\R^n$ with $\|y_i\|,\,\|\wit y_i\|\leq \rho\,,\; i=1,2$;
\end{itemize}
$\partial/\partial n$ denotes the
outward normal
derivative at the  boundary; and the boundary conditions are called Dirichlet boundary conditions if $\delta_i=0$ and $\alpha_i\equiv 1$, Neumann  boundary conditions if $\delta_i=1$ and $\alpha_i\equiv 0$,
or Robin boundary conditions if $\delta_i=1$ and $\alpha_i\geq 0$ is sufficiently regular on $\partial U\!$, for $i=1,\ldots,n$.
\par
Let  $C(\bar U)$ be the space of continuous real maps on the closure of $U$, endowed with
the sup-norm, which we just denote by $\|\,{\cdot}\,\|$. If every $\delta_i=1$, that is, with Neumann or Robin boundary conditions, the initial value $\varphi_i$ lies in the space $C([-1,0]\times \bar U)\equiv C([-1,0],C(\bar U))$ of the continuous maps on $[-1,0]$ taking values in $C(\bar U)$, whereas with Dirichlet boundary conditions $\varphi_i$ should in addition satisfy the compatibility condition $\varphi_i(0)\in C_0(\bar U)$, the subspace of $C(\bar U)$ of functions vanishing on $\partial U$.
\par
The former family can be written for short for $y(t,x)=(y_1(t,x),\ldots,y_n(t,x))$ as
\begin{equation}\label{family}
\left\{\begin{array}{l} \des\frac{\partial y}{\partial t}(t,x)=
 D\Delta y(t,x)+f(\w{\cdot}t,x,y(t,x),y(t-1,x))\,,\quad t>0\,,\;\,x\in \bar  U,\\[.2cm]
\bar\alpha(x)\,y(t,x)+\delta\,\des\frac{\partial y}{\partial n}(t,x) =0\,,\quad  t>0\,,\;\,x\in \partial U,\\[.2cm]
 y(s,x)=\varphi(s,x)\,,\quad  s\in [-1,0]\,,\;\,x\in \bar U,
\end{array}\right.
\end{equation}
for each $\w\in\W$, where  $D$ and $\bar\alpha(x)$ respectively stand for the $n\times n$ diagonal
matrices  with entries  $d_1,\ldots,d_n$ and $\alpha_1(x),\ldots,\alpha_n(x)$; $\delta=1$ for  Neumann or Robin boundary conditions and $\delta=0$ for Dirichlet boundary conditions; and $\varphi$ is a given map in the space $C([-1,0],C(\bar
U,\R^n))$, which can be identified with $C([-1,0]\times\bar
U,\R^n)$.
\par
Using results by Martin and Smith~\cite{masm0,masm} and  Travis and Webb~\cite{trwe}, the construction of a locally defined continuous skew-product semiflow linked to time-dependent families of IBV  problems given by systems of parabolic PFDEs with (possibly variable) finite delay has been explained in Novo et al.~\cite{nonuobsa} in the case of Neumann boundary conditions. In fact, the problem with Robin boundary conditions admits a common treatment. Notwithstanding, the problem with Dirichlet boundary conditions is more delicate.
\par
In any case, the main idea is to immerse the family of problems with delay~\eqref{family} into a family of retarded abstract equations in an appropriate Banach space $B$,
\begin{equation}\label{ACPdelay}
\left\{\begin{array}{l} z'(t)  =
 A  z(t)+ F(\w{\cdot}t,z_t)\,,\quad t>0\,,\\
z_0=\varphi\in C([-1,0],B)\,,
\end{array}\right.
\end{equation}
where for each $t\geq 0$, $z_t$ is the map defined by $z_t(s)=z(t+s)$ for $s\in [-1,0]$, and then use the semigroup theory approach. On the  space
of continuous functions $C([-1,0],B)$  the sup-norm will be used and it will be
denoted by $\n{\cdot}_C$.
\par
When it's time to choose a Banach space, it is important to have in mind the kind of results that one wants to obtain. On the one hand, we want a strongly continuous semigroup of operators, so that the induced skew-product semiflow is continuous. On the other hand, in Section~\ref{sec-uniform persistence} we will be working with some strong  monotonicity conditions, so that we need a cone of positive elements in the Banach space with a nonempty interior. For this reason, in the Dirichlet case we skip to work in $C_0(\bar U)$, since the natural cone of positive elements has an empty interior, and we better choose an intermediate space; more precisely, a domain of fractional powers associated to the realization of the Dirichlet Laplacian in $L^p(U)$. Nice sections dedicated to these spaces can be found in Henry~\cite{henr}, Lunardi~\cite{luna} or Pazy~\cite{pazy}.
\par
At this point it seems convenient to present
 the Neumann and Robin cases, and the Dirichlet case separately.
\subsection{The case of Neumann or Robin boundary conditions}\label{sect-Neumann}
In this case for each component $i=1,\ldots,n$ we consider on the space $C(\bar U)$ the differential operator $A_i^0 z_i = d_i\Delta z_i$ with domain $D(A_i^0)$ given by
\[
\left\{z_i\in C^2(U)\cap C^1(\bar U)\;\Big|\; A_i^0z_i\in C(\bar U)\,,\;
\alpha_i(x)\,z_i(x)+\des\frac{\partial z_i}{\partial n}(x)=0\, \;\forall \;x\in \partial U \right\}.
\]
Then, the closure $A_i$ of $A_i^0$ in $C(\bar U)$ is a sectorial operator and it generates an
analytic semigroup of bounded  linear operators $(T_i(t))_{t\geq 0}$, which is usually just written down as  $(e^{tA_i})_{t\geq 0}$,
and $e^{tA_i}$ is compact for any $t>0$ (for instance, see Smith~\cite{smit}). Besides, the semigroup is strongly continuous, that is, $A_i$ is densely defined.
%More precisely, %(citar Stewart ???? Lunardi le cita en sus notas p\'{a}g 38-39),
%the domain $D(A_i)$ of the sectorial operator $A_i$ is given by
%\[
%\left\{z_i\in \cap_{1\leq p<\infty} W^{2,p}(U)\;\Big|\; z_i\,,\,d_i\Delta z_i\in C(\bar U)\,,\;
%\alpha_i(x)\,z_i(x)+\des\frac{\partial z_i}{\partial n}(x)=0\, \;\forall \;x\in \partial U \right\}.
%\]
%Recall that $W^{k,p}(U)$ is the Sobolev space of all the functions $h$ in $L^p(U)$ which admit weak derivatives $D^\alpha h$, up to order $k$,
%belonging to $L^p(U)$. %The space $W^{k,p}(U)$ is endowed with the norm
%\[
%\|h\|_{W^{k,p}(U)}=\sum_{|\alpha|\leq k} \|D^{\alpha}h\|_p\,.
%\]
\par
On the product Banach space  $E=C(\bar U)^n\equiv C(\bar
U,\R^n)$ endowed with the norm $\|(z_1,\ldots,z_n)\|=\sum_{i=1}^n\|z_i\|$, we consider the operator $A=\Pi_{i=1}^n  A_i$ with domain $D(A)=\Pi_{i=1}^n D(A_i)$, which is sectorial and generates an analytic semigroup of operators $(e^{tA})_{t\geq 0}$, with  $e^{tA}=\Pi_{i=1}^n e^{tA_i}$, and $e^{tA}$ is compact for any $t>0$.
\par
Let us define $F:\W\times C([-1,0],E)\to E$, $(\w,\varphi)\mapsto F(\w,\varphi)$ by
\begin{equation}\label{F}
F(\w,\varphi)(x)=f(\w,x,\varphi(0,x),\varphi(-1,x))\,,\quad x\in \bar U
\end{equation}
and consider the retarded abstract problems on $E$ given in~\eqref{ACPdelay}. As explained in Novo et al.~\cite{nonuobsa}, with condition (C) on $f$, mild solutions of these ACPs with delay, that is, continuous solutions of the integral equations
\begin{equation}\label{variation constants}
z(t)=e^{tA}\,z(0) +\int_0^t e^{(t-s)A}\,F(\w{\cdot}s,z_s)
\,ds\,,\quad t\geq 0\,,
\end{equation}
permit us to set a locally defined continuous skew-product semiflow
\begin{equation*}
\begin{array}{cccc}
 \tau: &\mathcal{U} \subseteq\R_+\times \W\times C([-1,0],E) & \longrightarrow & \W\times C([-1,0],E)\\
 & (t,\w,\varphi) & \mapsto
 &(\w{\cdot}t,z_t(\w,\varphi))\,,
\end{array}
\end{equation*}
for an appropriate open set $\mathcal{U}$, where as usual  $z_t(\w,\varphi)(s)=z(t+s,\w,\varphi)$ for every $s\in [-1,0]$, for $t\geq 0$. Besides, for any $t>1$ the section map $\tau_t$ is compact, meaning that it takes bounded sets in $\W\times C([-1,0],E)$ into relatively compact sets (see Proposition~2.4 in Travis and Webb~\cite{trwe}), and  if a solution $z(t,\w,\varphi)$ remains bounded, then it is defined on the whole positive real line and the semiorbit of $(\w,\varphi)$ is relatively compact.%: see the proof of Proposition~3.1 in Novo et al.~\cite{nonuobsa} for similar arguments (in fact inspired in the proof of Proposition~2.4 in Travis and Webb~\cite{trwe}) and note that the compact character of the operators $e^{tA}\in\mathcal{L}(E)$ for $t>0$ is crucial here.
\par
It is well-known that we have to impose some extra conditions on the map $f$ in~\eqref{family} in order to gain regularity in the solutions of the associated ACPs~\eqref{ACPdelay}. For completeness, we include the definition of what we call a classical solution. Different names for the same concept are sometimes found in the literature.
\begin{defi}\label{defi-clasica}
A map $z\in C^1((0,T],E)\cap C((0,T],D(A))\cap C([0,T],E)$ which satisfies~\eqref{ACPdelay} for $0<t\leq T$  is a {\em classical solution\/} on $[0,T]$.
\end{defi}
The following condition is often referred to as a {\em time regularity\/} condition.
\begin{itemize}
\item[($C^\theta(t)$)] $f(\w{\cdot}t,x,y,\wit y)$ is $\theta$-H\"{o}lder continuous in $t$ (for some $\theta\in(0,1)$) in bounded sets of $\R^n\times \R^n$ uniformly for  $\w\in \W$ and $x\in\bar U$; that is, given any $r>0$ there exists an $l_r>0$ such that
\[
\|f(\w{\cdot}t,x,y,\wit y)-f(\w{\cdot}s,x,y,\wit y)\|\leq l_r\,|t-s|^\theta\,,\quad t,\,s\geq 0\,,
\]
for any $\w\in \W$, $x\in\bar U$ and $y,\wit y\in\R^n$ with $\|y\|, \|\wit y\|\leq r$\,.
\end{itemize}
\par
We include a short proof of the following result, which follows from Theorem~4.1 in Novo et al.~\cite{nonuobsa}.
\begin{teor}\label{teor-time regularity-Neumann}
Assume conditions $(C)$ and $(C^\theta(t))$, for some $\theta\in(0,1/2)$, on the map $f$ in~\eqref{family}.
Then, for fixed $\w\in\W$ and $\varphi\in C([-1,0],E)$:
\begin{itemize}
\item[(i)] The mild solution of~\eqref{ACPdelay} is classical for $t\geq 1$, provided that it is defined.
\item[(ii)] If $\varphi:[-1,0]\to E$ is $\theta$-H\"{o}lder continuous and besides $\varphi(0)\in C^{2\theta}(\bar  U,\R^n)$,
then the mild solution of~\eqref{ACPdelay} is a classical solution on intervals $[0,T]$ as long as it is defined.
\end{itemize}
\end{teor}
\begin{proof}
(i) Assume that the mild solution $z(t)=z(t,\w,\varphi)$ of~\eqref{ACPdelay} is defined for $t\in [0,T]$, and let $g(t)=F(\w{\cdot}t,z_t)$ for $t\in [0,T]$. Then, if $T>1$, for any $\varepsilon>0$, it is well-known that $z\in C^\theta([\varepsilon,T],E)$ meaning that it is $\theta$-H\"{o}lder continuous in $t$ (see Lunardi~\cite{luna}), so that under conditions (C) and $(C^\theta(t))$, $g$ is $\theta$-H\"{o}lder continuous on $[1+\varepsilon,T]$.  The classical theory for the nonhomogeneous equation $z'(t)=Az(t)+g(t)$ then says that $z(t)$ is a classical solution on $[1,T]$ (see Henry~\cite{henr} or Lunardi~\cite{luna}).
\par
For (ii), note that with Neumann of Robin boundary conditions, $z\in C^\theta([0,\varepsilon],E)$ if and only if $\varphi(0)\in C^{2\theta}(\bar  U,\R^n)$, provided that $\theta<1/2$ (for instance, see Lunardi~\cite{luna}), and then just argue as before.
\end{proof}
Still an additional condition has to be imposed on $f$ in order to have classical solutions $y(t,x)$ of the IBV problems with delay~\eqref{family}:
\begin{itemize}
\item[$(C^\theta(x))$] $f(\w,x,y,\wit y)$ is $\theta$-H\"{o}lder continuous in $x$ (for some $\theta\in(0,1)$) in bounded sets of $\R^n\times\R^n$ uniformly for  $\w\in \W$; that is, given any $r>0$ there exists an $l_r>0$ such that for any $\w\in \W$ and $y,\wit y\in\R^n$ with $\|y\|, \|\wit y\|\leq r$,
\[
\|f(\w,x_2,y,\wit y)-f(\w,x_1,y,\wit y)\|\leq l_r\,\|x_2-x_1\|^\theta\,,\quad x_1,\,x_2\in \bar U.
\]
\end{itemize}
\par
Note that the classical space where one looks for solutions is $C^{1,2}([a,b]\times \bar U,\R^n)$, for appropriate time intervals $[a,b]$. We are going to use some optimal regularity results of solutions of IBV problems contained in Lunardi~\cite{luna}. Nevertheless, since we are just interested in the  $C^{1,2}$ regularity of solutions, we are not going to pay the due attention to the optimal regularity there proved. %Also, note that many results in~\cite{luna} are given for problems $\partial u/\partial t=Lu+h(t,x)$ with more general second order  strongly elliptic operators $L$ than the Laplacian.
Some classical references for regularity results are Friedman~\cite{frie} and Ladyzhenskaja et al.~\cite{lasu}.  % who basically proves that the solution of $\partial u/\partial t=Lu+h(t,x)$ on domains is as smooth as the coefficients in the equation, for more general second order  strongly elliptic operators $L$ than the Laplacian.
\begin{teor}\label{teor-space regularity-Neumann}
Assume conditions $(C)$, $(C^{\theta}(t))$ and $(C^{2\theta}(x))$ on the map $f$ in~\eqref{family}, for some $\theta\in(0,1/2)$, with Neumann or Robin boundary conditions. For  fixed $\w\in\W$ and $\varphi\in C([-1,0],E)$, assume that the mild solution $z(t)=z(t,\w,\varphi)$ is defined on a time interval $[0,T]$  and set $y(t,x)=z(t)(x)$ for $t\in [0,T]$ and $x\in \bar U$, as well as $y(s,x)=\varphi(s,x)$ for $s\in [-1,0]$ and $x\in \bar U$. Then:
\begin{itemize}
\item[(i)] If $T>1$, for any $\varepsilon>0$ the map $y\in C^{1,2}([1+\varepsilon,T]\times \bar U,\R^n)$ is a solution of the IBV problem~\eqref{family} for $1+\varepsilon<t\leq T$.
\item[(ii)] If  $\varphi\in C^{\theta,2\theta}([-1,0]\times \bar U, \R^n)$,
then, for any $\varepsilon>0$,  $y\in C^{1,2}([\varepsilon,T]\times \bar U,\R^n)$ is a solution of the IBV problem~\eqref{family} for $\varepsilon<t\leq T$.
\end{itemize}
\end{teor}
\begin{proof}
For the continuous map $h(t,x)=f(\w{\cdot}t,x,y(t,x),y(t-1,x))$, $(t,x)\in [0,T]\times\bar U$, we consider the IBV problem on $[0,T]\times\bar U$,
\begin{equation}\label{auxiliar-robin}
\left\{\begin{array}{l} \des\frac{\partial y}{\partial t}(t,x)=
 D\Delta y(t,x)+h(t,x)\,,\quad 0<t\leq T,\;\,x\in \bar U,\\[.2cm]
\bar\alpha(x)\,y(t,x)+\des\frac{\partial y}{\partial n}(t,x) =0\,,\quad   0<t\leq T,\;\,x\in \partial U,\\[.2cm]
 y(0,x)=\varphi(0,x)\,, \quad x\in \bar U.
\end{array}\right.
\end{equation}
\par
Note that in both items $y(t,x)=z(t)(x)$ is $C^1$ in $t$ because $z(t)$ is a classical solution, by   Theorem~\ref{teor-time regularity-Neumann}. Therefore, it remains to check the $C^2$ regularity in $x$.
\par
By Theorem~5.1.17 in~\cite{luna}, $y(t,x)\in C^{\theta,2\theta}([\delta,T]\times \bar U)$ for any $\delta>0$  (condition (C) on $f$ is enough for this). Then, fixed $\varepsilon>0$, $h(t,x)\in C^{\theta,2\theta}([1+\frac{\varepsilon}{2},T]\times \bar U)$, and from Theorem~\ref{teor-time regularity-Neumann},   $z(1+\frac{\varepsilon}{2})\in D(A)$, so that the boundary condition is fulfilled at $t=1+\frac{\varepsilon}{2}$.
Then, we can apply Proposition~7.3.3~(iii) in~\cite{luna} to the IBV problem~\eqref{auxiliar-robin} for $1+\frac{\varepsilon}{2}<t\leq T$ with initial condition $y(1+\frac{\varepsilon}{2},x)=z(1+\frac{\varepsilon}{2})(x)$ for  $x\in \bar U$,
 to deduce that $y\in C^{1,2}([1+\varepsilon,T]\times \bar U,\R^n)$. The proof of (i) is finished. %On this occasion,  we argue as the proof of Theorem~3.1~(b) in Smith~\cite{smit}, where Corollary~2 to Theorem~2 in Chapter~5 in Friedman~\cite{frie} is used
\par
Recall that $\varphi(0)\in C^{2\theta}(\bar U,\R^n)$ is the necessary and sufficient condition to guarantee that $y(t,x)\in C^{\theta,2\theta}([0,T]\times \bar U)$. With the assumption $\varphi\in C^{\theta,2\theta}([-1,0]\times \bar U, \R^n)$ in~(ii),  $h(t,x)\in C^{\theta,2\theta}([0,T]\times \bar U)$. Arguing as in the previous paragraph,  we can deduce that $y\in C^{1,2}([\varepsilon,T]\times \bar U,\R^n)$ for any $\varepsilon>0$. The proof is finished.
\end{proof}
%%%%%%%%%%%%%%%%%%%%%%%%%%%%%%%%%%%%%%%%%%%%%%%%%%%%%%%%%%%%%%%%%%%%%%%%%%%
%%%%%%%%%%%%%%%%%%%%%%%%%%%%%%%%%%%%%%%%%%%%%%%%%%%%%%%%%%%%%%%%%%%%%%%%%%%%%%%%
\subsection{The case of Dirichlet boundary conditions}\label{sect-Dirichlet}
This time, for each component $i=1,\ldots,n$ we consider on the space $C(\bar U)$ the differential operator $A_i^0 z_i = d_i\Delta z_i$ with domain $D(A_i^0)=\{z_i\in C^2(U)\cap C_0(\bar U)\;|\; A_i^0z_i\in C_0(\bar U)\}$.
The closure $A_i$ of $A_i^0$ in $C(\bar U)$ is a sectorial operator which generates an
analytic semigroup of bounded  linear operators $(e^{tA_i})_{t\geq 0}$, with  $e^{tA_i}$ compact for any $t>0$ (see Lunardi~\cite{luna} or Smith~\cite{smit}), but now the semigroup is not strongly continuous, since
%$A_i$ is not densely defined. Precisely, %(citar Stewart ???? Lunardi le cita en sus notas p\'{a}g 38-39),
%\[
%D(A_i)=\{z_i\in \cap_{1\leq p<\infty} W^{2,p}(U)\;|\; z_i\,,\,d_i\Delta z_i\in C(\bar U)\,,\;
%{z_i}_{|\partial U}=0\}\,,
%\]
%and
$\overline{ D(A_i)}=C_0(\bar U)$. As in Section \ref{sect-Neumann},  $A=\Pi_{i=1}^n  A_i$ is the sectorial operator with domain $D(A)=\Pi_{i=1}^n D(A_i)$ on the product Banach space  $E=C(\bar U)^n\equiv C(\bar
U,\R^n)$, and $e^{tA}$ is compact for any $t>0$.
\par
In this case, we also consider for each component $i$ the realization of the Dirichlet $d_i$-Laplacian on the Banach space $L^p(U)$ for a fixed $m<p<\infty$, that is, the operator $A_{i,p}:D(A_{i,p})\subset L^p(U)\to L^p(U)$ %with domain $D(A_{i,p})=W^{2,p}(U)\cap W^{1,p}_0(U)$,
defined by $A_{i,p}z_i=d_i\Delta z_i$ (in a weak sense) for $z_i\in D(A_{i,p})$. This operator is sectorial, densely defined and $0\in\rho(A_{i,p})$. Then, for $\alpha\in (1/2+m/(2p),1)$, let $E_i^\alpha:=D(-A_{i,p})^\alpha=\rg (-A_{i,p})^{-\alpha}$ be the domain of fractional power $\alpha$ of $-A_{i,p}$, which is a Banach space with norm
$\|z_i\|_\alpha = \|(-A_{i,p})^{\alpha}\,z_i\|_p$
and satisfies $E_i^\alpha\hookrightarrow C^1(\bar U)$ (see Theorem~1.6.1 in Henry~\cite{henr}). Besides,
$E_i^\alpha$ is an intermediate space in the class $J_\alpha$ between $L^p(U)$ and $D(A_{i,p})$, that is, we have continuous embeddings
$D(A_{i,p}) \hookrightarrow E_i^\alpha \hookrightarrow L^p(U)$
and there exists a constant $c_i>0$ such that $
\|z_i\|_\alpha\leq c_i\,\|A_{i,p}z_i\|_p^\alpha\,\|z_i\|_p^{1-\alpha}$ for any $z_i\in D(A_{i,p})$.
Also the following estimate holds, which will be used later on:
\begin{equation}\label{estimate}
\| (-A_{i,p})^\alpha\,e^{t A_{i,p}}\|_{\mathcal{L}(L^p(U))}\leq M_\alpha\,t^{-\alpha}\,e^{-w t}\,,\quad t>0\, ,
\end{equation}
for some $w>0$ and $M_\alpha>0$ (see Theorem~6.13 in Pazy~\cite{pazy}).
\par
In all, $D(A_i)\hookrightarrow D(A_{i,p})\hookrightarrow E_i^\alpha \hookrightarrow C^1(\bar U) \hookrightarrow C(\bar U) \hookrightarrow L^p(U)$
and $E_i^\alpha$ as an intermediate space  in the class $J_\alpha$ between $C(\bar U)$ and $D(A_i)$.
\par
This time, we consider on the product Banach space $L^p(U)^n\equiv L^p(
U,\R^n)$ with norm $\|(z_1,\ldots,z_n)\|_p=\sum_{i=1}^n \|z_i\|_p$, the operator $A_p=\Pi_{i=1}^n  A_{i,p}$ with domain $D(A_p)=\Pi_{i=1}^n D(A_{i,p})$ and the bounded linear operator $(-A_p)^{-\alpha}=\Pi_{i=1}^n (-A_{i,p})^{-\alpha}$. We also consider the product Banach space $E^\alpha=\Pi_{i=1}^n  E_i^\alpha$ endowed with the norm $\|(z_1,\ldots,z_n)\|_\alpha=\sum_{i=1}^n\|z_i\|_\alpha$. Thanks to H\"{o}lder's inequality, it is immediate to check that $E^\alpha$ is an intermediate Banach space between $E$ and $D(A)$ in the class $J_\alpha$.
In fact, because of the continuous embeddings, $(e^{t A})_{t\geq 0}$ is an analytic semigroup of bounded linear operators on $E^\alpha$ and besides in this case:
\begin{equation}\label{sup finito}
\limsup_{t\to 0^+} \|e^{tA}\|_{\mathcal{L}(E^\alpha)}<\infty\,.
\end{equation}
Furthermore, it is easy to check that the semigroup of operators $(e^{t A})_{t\geq 0}$ is strongly continuous on $E^\alpha$, that is, $D(A)$ is dense in $E^\alpha$: just take any $z\in E^\alpha$, that is, $z=(-A_{p})^{-\alpha}\,y$ for some $y\in L^p(U,\R^n)$ and since $(-A_{p})^{-\alpha}$ commutes with $e^{tA_{p}}$, $\|e^{t A}z-z\|_\alpha= \|e^{t A_p}(-A_{p})^{-\alpha}\,y- (-A_{p})^{-\alpha}\,y\|_\alpha=\|e^{t A_p}y-y\|_p\to 0$
as $t\to 0^+$, since $(e^{t A_p})_{t\geq 0}$ is strongly continuous in $L^p(U,\R^n)$. In particular, $E^\alpha\hookrightarrow\overline{ D(A)}=C_0(\bar U,\R^n)$.
Finally,  $e^{t A}:E^\alpha \to E^\alpha$ is compact for any $t>0$. This follows from $E^\alpha \hookrightarrow E$, the compactness of $e^{(t/2) A}:E\to E$ and the boundedness of $e^{(t/2) A}:E\to E^\alpha$ because $E^\alpha$ is an intermediate space in the class $J_\alpha$.
\par
On this occasion, we consider $F:\W\times C([-1,0],E^\alpha)\to E$, $(\w,\varphi)\mapsto F(\w,\varphi)$ defined as in~\eqref{F}, and the retarded ACPs on $E^\alpha$ given in~\eqref{ACPdelay}. Although there are some results for these problems in the $\alpha$-norm (e.g., see Travis and Webb~\cite{trwe78}),  here we opt to apply the ``method of steps" to get existence and uniqueness of mild solutions of~\eqref{ACPdelay}, arguing on $[0,1]$ first, then on $[1,2]$, and so on. In this way we can apply the well-established theory for semilinear ACPs with nonlinearities defined in intermediate spaces (for instance, see Chapter~7 in Lunardi~\cite{luna}). So, for fixed $\w\in \W$ and $\varphi\in C([-1,0],E^\alpha)$, let us define the map $
\wit F:[0,1]\times E^\alpha \to E$
by $\wit F(t,v)(x)=f(\w{\cdot}t,x,v(x),\varphi(t-1,x))$ for any $x\in
\bar U$, for the map $f$ in~\eqref{family}. It is easy to check that condition (C) on $f$ is transferred to the map $\wit F$, in the sense that $\wit F$ is continuous and it is Lipschitz in $v$ in bounded sets of $E^\alpha$, uniformly for $t\in [0,1]$; that is, given $R>0$ there exists a $C_R>0$ such that for $t\in[0,1]$,
\begin{equation}\label{lipsch v}
\|\wit F(t,v_2)-\wit F(t,v_1)\|\leq C_R\,\|v_2-v_1\|_\alpha\quad\hbox{for any }\; \|v_1\|_\alpha, \|v_2\|_\alpha\leq R\,.
\end{equation}
\par
With these conditions, the standard theory for the semilinear ACP in $E^\alpha$,
\begin{equation}\label{ACP}
\left\{\begin{array}{l} z'(t)  =
 A  z(t)+ \wit F(t,z(t))\,,\quad t>0\,,\\
z(0)=\varphi(0)\in  E^\alpha\,,
\end{array}\right.
\end{equation}
with $A$ sectorial and densely defined, says that the problem admits a unique mild solution $z=z(t,\w,\varphi)\in C([0,\delta],E^\alpha)$ for a certain  $\delta=\delta(\w,\varphi)\in (0,1]$, that is, $z$ is a continuous solution of the integral equation
\begin{equation*}
z(t)=e^{tA}\,z(0) +\int_0^t e^{(t-s)A}\,\wit F(s,z(s))
\,ds\,,\quad t\in [0,\delta]\,.
\end{equation*}
Compare with~\eqref{variation constants} to see that $z$ is also a mild solution of~\eqref{ACPdelay} on  $[0,\delta]$.
\par
Whenever the mild solution is globally defined on $[0,1]$, then we consider the ACP~\eqref{ACP} on $[1,2]$ with
$\wit F(t,v)(x)=f(\w{\cdot}t,x,v(x),z(t-1,\w,\varphi)(x))$ for any $x\in \bar U$, %with $z(s,\w,\varphi)(x)=\varphi(s,x)$ for $s\in [-1,0]$.
and $z(1)=z(1,\w,\varphi)\in  E^\alpha$. Now $\wit F$ is continuous and satisfies~\eqref{lipsch v} on $[1,2]$, and once more the problem admits a unique mild solution. Both solutions stuck together give the mild solution on $[0,1+\delta']$, and note once more that $z$ is a mild solution of~\eqref{ACPdelay} too. This procedure can be iterated, as long as the mild solution is defined.
\par
Standard arguments using  a generalized version of the Gronwall's lemma (see Lemma~7.1.1 in Henry~\cite{henr}) permit to see that the mild solution $z(t,\w,\varphi)$ depends continuously on the initial condition $\varphi$, and also on $\w\in\W$ (note that the map $F$ depends on both $\w$ and  $\varphi$).
Therefore,  mild solutions of the ACPs permit us to set a locally defined continuous skew-product semiflow
\begin{equation*}
\begin{array}{cccc}
 \tau: &\mathcal{U} \subseteq\R_+\times \W\times C([-1,0],E^\alpha) & \longrightarrow & \W\times C([-1,0],E^\alpha)\\
 & (t,\w,\varphi) & \mapsto
 &(\w{\cdot}t,z_t(\w,\varphi))\,,
\end{array}
\end{equation*}
for an appropriate open set $\mathcal{U}$, where $z_t(\w,\varphi)(s)=z(t+s,\w,\varphi)$ for every $s\in [-1,0]$. Also here, if a solution $z(t,\w,\varphi)$ remains bounded, then it is defined on the whole positive real line and the semiorbit of $(\w,\varphi)$ is relatively compact: see the arguments in the proof of
Proposition~3.1 in Novo et al.~\cite{nonuobsa}.
\par
Note that we can also consider $F:\W\times C([-1,0],E)\to E$ defined as in~\eqref{F}, and solve the retarded ACP~\eqref{ACPdelay} for any $\varphi\in C([-1,0],E)$ (with $\varphi(0)\in C_0(\bar U,\R^n)$ if continuity of the mild solution up to $t=0$ is wanted). Then, since $E^\alpha$ is an intermediate space between $E$ and $D(A)$, from~\eqref{variation constants} it follows that  for $t>1$, $z_t(\w,\varphi)\in C([-1,0],E^\alpha)$ (see Proposition~4.2.1 in Lunardi~\cite{luna}).
In fact, one can prove that for $t>1$, the section semiflow $\tau_t:\W\times C([-1,0],E)\to \W\times C([-1,0],E^\alpha)$ is compact on its domain: argue as in Proposition~2.4 in Travis and Webb~\cite{trwe}.
\par
We finish this section with some results on regularity of solutions. A classical solution is defined  exactly  as in Definition~\ref{defi-clasica}. %Note that if condition ($C^\theta(t)$) holds on a compact time interval, then ($C^{\theta'}\!(t)$) holds too for any $\theta'<\theta$ on that interval.
\begin{teor}\label{teor-time regularity-Dirichlet}
Assume conditions $(C)$ and $(C^\theta(t))$, for some $\theta\in(0,1/2)$, on the map $f$ in~\eqref{family}.
Then, for fixed $\w\in\W$ and $\varphi\in C([-1,0],E^\alpha)$:
\begin{itemize}
\item[(i)] The mild solution of~\eqref{ACPdelay} is classical for $t\geq 1$, provided that it is defined.
\item[(ii)] If $\varphi:[-1,0]\to E$ is $\theta$-H\"{o}lder continuous, then the mild solution of~\eqref{ACPdelay} is a classical solution on intervals  $[0,T]$ as long as it is defined.
\end{itemize}
\end{teor}
\begin{proof}
The proof follows the same lines as that of Theorem~\ref{teor-time regularity-Neumann}. Just recall that $E^\alpha\hookrightarrow\overline{ D(A)}=C_0(\bar U,\R^n)$, and that with Dirichlet boundary conditions the mild solution $z\in C^\theta([0,\varepsilon],E)$ if and only if $\varphi(0)\in C_0^{2\theta}(\bar U,\R^n)$ (see~\cite{luna}). Since $\varphi(0)\in E^\alpha\hookrightarrow C^1(\bar U,\R^n)$, and $\theta<1/2$,  we can give it for granted. The proof is finished.
\end{proof}
\begin{teor}\label{teor-space regularity-Dirich}
Assume conditions $(C)$, $(C^{\theta}(t))$ and $(C^{2\theta}(x))$ on the map $f$ in~\eqref{family},  for some $\theta\in(0,1/2)$, with Dirichlet boundary conditions. %plus the compatibility condition $f(\w,x,0,0)=0$ for any $\w\in \W$ and $x\in \partial U$.
For  fixed $\w\in\W$ and $\varphi\in C([-1,0],E^\alpha)$, assume that the mild solution $z(t)=z(t,\w,\varphi)$ is defined on a time interval $[0,T]$ and set $y(t,x)=z(t)(x)$ for $t\in [0,T]$ and $x\in \bar U$, as well as $y(s,x)=\varphi(s,x)$ for $s\in [-1,0]$ and $x\in \bar U$. Then:
\begin{itemize}
\item[(i)] If $T>1$, for any $\varepsilon>0$ the map $y\in C^{1,2}([1+\varepsilon,T]\times \bar U,\R^n)$ is a solution of the IBV problem~\eqref{family} for $1+\varepsilon<t\leq T$.
\item[(ii)] If  $\varphi\in C^\theta([-1,0], E)$,  %and  one has the additional compatibility condition $\varphi(s,x)=0$ for $s\in [-1,0]$ and $x\in \partial U$,
then for any $\varepsilon>0$ the map $y\in C^{1,2}([\varepsilon,T]\times \bar U,\R^n)$ is a solution of the IBV problem~\eqref{family} for $\varepsilon<t\leq T$. %If in addition $\varphi(0)\in C^{2}(\bar U,\R^n)$, then $y\in C^{1,2}([0,T]\times \bar U,\R^n)$.
\end{itemize}
\end{teor}
\begin{proof}
The proof follows the same lines as that of Theorem~\ref{teor-space regularity-Neumann}.
For the continuous map $h(t,x)=f(\w{\cdot}t,x,y(t,x),y(t-1,x))$ for $(t,x)\in [0,T]\times\bar U$, consider the IBV problem~\eqref{auxiliar-robin}  but with boundary condition $y(t,x)=0$ for $0<t\leq T$, $x\in \partial U$.
\par
This time by Theorem~5.1.11 in~\cite{luna}, $y(t,x)\in C^{\theta,2\theta}([\varepsilon,T]\times \bar U)$ for any $\varepsilon>0$. Since we are working on $E^\alpha$ with $\alpha>1/2$,  $\varphi(0)\in E^\alpha\hookrightarrow C^{2\theta}_0(\bar U,\R^n)$ and then in fact $y(t,x)\in C^{\theta,2\theta}([0,T]\times \bar U)$.
Therefore, the map $h(t,x)\in C^{\theta,2\theta}([1,T]\times \bar U)$. Besides, since for any $t\geq 0$, $z(t)\in E^\alpha \hookrightarrow C_0(\bar U,\R^n)$, we have  that $z(t)(x)=0$ for any $t\geq 0$ and any $x\in \partial U$. %In particular this implies that $h(t,x)=0$ for any $(t,x)\in [1,T]\times \partial U$.
With these conditions we can apply Proposition~7.3.2~(iii) in~\cite{luna} %Theorem~5.1.13 in~\cite{luna}
to the IBV problem for $1<t\leq T$ with initial condition $y(1,x)=z(1)(x)$ for  $x\in \bar U$, to get that $y\in C^{1,2}([1+\varepsilon,T]\times \bar U,\R^n)$ and (i) is proved.
\par
With the assumptions in (ii), and the fact that $\varphi(s)\in E^\alpha\hookrightarrow C^{2\theta}(\bar U,\R^n)$ for any $s\in [-1,0]$, now  $h(t,x)\in C^{\theta,2\theta}([0,T]\times \bar U)$. %and $h(t,x)=0$ for any $(t,x)\in [0,T]\times \partial U$.
 Once more Proposition~7.3.2~(iii) in~\cite{luna} implies that $y\in C^{1,2}([\varepsilon,T]\times \bar U,\R^n)$.  The proof is finished.
\end{proof}
%%%%%%%%%%%%%%%%%%%%%%%%%%%%%%%%%%%%%%%%%%%%%%%%%%%%%%%%%%%%%%%%%%%%%%%%%%%%%%%%%
%%%%%%%%%%%%%%%%%%%%%%%%%%%%%%%%%%%%%%%%%%%%%%%%%%%%%%%%%%%%%%%%%%%%%%%%%%%%%%%%%%%%%%%%
\subsection{Monotone skew-product semiflows induced by quasimonotone parabolic PFDEs}\label{sec-monotonos}
In this section we are concerned with the classical quasimonotone condition which renders the skew-product semiflow induced by mild solutions monotone. We state this result, together with a technical inequality which will be fundamental in Section~\ref{sec-uniform persistence}.
\par
First of all, we describe the cones of positive vectors in the spaces we are dealing with. In the case of Neumann or Robin boundary conditions, $C([-1,0],E)$ is a strongly ordered Banach space
with positive cone $
C_+([-1,0],E)=\{\varphi\in C([-1,0],E)\mid \varphi(s)\in
E_+\;\text{for} \; s\in [-1,0]\}$
where $E_+=\{z\in E\mid z(x)\geq 0\;\text{for} \;
x\in\bar U\}$ and  $\R^n_+=\{y\in\R^n\mid y_i\geq 0\; \text{for}\;
i=1,\ldots,n\}$. Note that we can trivially identify
\[
\Int
C_+([-1,0],E)=\{\varphi \in C([-1,0]\times\bar U,\R^n)\mid \varphi(s,x)\gg 0
\;\text{for}\; s\in [-1,0]\,,\,x\in\bar U \}\,.
\]
\par
In the case of  Dirichlet boundary conditions, $C([-1,0],E^\alpha)$ is a strongly ordered Banach space with
$
C_+([-1,0],E^\alpha)=\{\varphi\in C([-1,0],E^\alpha)\mid \varphi(s)\in
E^\alpha_+\;\text{for} \; s\in [-1,0]\}\,,
$
where the positive cone in $E^\alpha$ is $E^\alpha_+=\{z\in E^\alpha\,\big|\; z(x)\geq 0\;\text{for} \;
x\in\bar U \}$. %Here recall that $E^\alpha \hookrightarrow C^1(\bar U,\R^n)$.
Besides, $E^\alpha_+$ has a nonempty interior, since
\begin{equation*}
\Big\{z\in E^\alpha_+\,\big|\; z(x)\gg 0\;\text{for} \;
x\in U \;\text{and}\; \frac{\partial z}{\partial n}(x)\ll 0 \;\text{for} \;
x\in\partial U   \Big\}= \Int E^\alpha_+\,,
\end{equation*}
and $\Int C_+([-1,0],E^\alpha)=\{\varphi\in C([-1,0],E^\alpha)\mid \varphi(s)\in \Int E^\alpha_+ \text{ for }s\in[-1,0]\}\not=\emptyset$.
\par
To unify the writing, $E^\gamma$ will stand for the Banach space $E$ in the problem with Neumann or Robin boundary conditions, and for $E^\alpha$ in the problem with Dirichlet boundary conditions; and $C_\gamma$ for the space $C([-1,0],E^\gamma)$ with the sup-norm $\|\,{\cdot}\,\|_{C_\gamma}$. Also, the order relations in $E^\gamma$ and $C_\gamma$ will just be denoted by $\leq$, $<$ and $\ll$ according to~\eqref{order}, but have in mind the different spaces involved in each case.
\begin{prop}\label{prop-monotone}
Assume hypotheses $\rm{(C)}$ and $(C^{\theta}(t))$, for some $0<\theta<1/2$, on the
map $f$ in~\eqref{family}, plus the quasimonotone condition:
\begin{itemize}
\item[(QM)] If $\,y,\wit y, u,\wit u\in \R^n$ with $y\leq u$, $\wit y\leq \wit u$ and $y_i=u_i$ for some $i\in\{1,\ldots,n\}$, then
$f_i(\w,x,y,\wit y)\leq f_i(\w,x,u, \wit u)$ for any $\w\in\W$ and $x\in\bar U$.
\end{itemize}
Besides, in the Dirichlet case assume further:
\begin{itemize}
\item[(DM)] $f(\w,x,0,0)=0$ for any $\w\in\W$ and $x\in\partial U$.
\end{itemize}
Then:
\begin{itemize}
\item[(i)] The induced skew-product semiflow on $\W\times C_\gamma$
is monotone, that is, if $\varphi,\psi\in
C_\gamma$ with $\varphi\leq \psi$, then $z_t(\w,\varphi)\leq
z_t(\w,\psi)$ for any $\w\in\W$ and any $t\geq 0$ where both terms
are defined.
\item[(ii)] Given $\w\in\W$ and $\varphi,\psi\in
C_\gamma$ with $\varphi\leq \psi$ such that $z(t,\w,\varphi)$ and $z(t,\w,\psi)$ are defined for
$t\in[0,\beta]$ for some $\beta>0$, there exists an $L=L(\w,\varphi,\psi,\beta)>0$ such that
for each $i=1,\ldots,n$, and for each $t\in[0,\beta]$,
\[
z_i(t,\w,\psi)-z_i(t,\w,\varphi)\geq e^{-L
t}\,e^{t A_i}\,(\psi_i(0)-\varphi_i(0))\,.
\]
\end{itemize}
\end{prop}
\begin{proof}
(i) For each fixed $\w\in \W$, it follows from the results in Martin and Smith~\cite{masm0,masm}.
\par
(ii)  We include the proof for the sake of completeness, although the result in the Neumann case follows from Lemma~4.3 in Novo et al.~\cite{nonuobsa}. First of all, observe that if  $\w\in\W$, and $\varphi,\psi\in
C_\gamma$ with $\varphi\leq \psi$ and $\|\varphi(s)(x)\|,\|\psi(s)(x)\|\leq \rho$ for any $s\in [-1,0]$ and $x\in \Bar U$, then, for any $t\in\R$ and $x\in\bar U$,
\begin{multline}\label{lemma 1.1}
f_i(\w{\cdot}t,x,\psi(0)(x),\psi(-1)(x))-f_i(\w{\cdot}t,x,\varphi(0)(x),\varphi(-1)(x))\geq \\ - L\,(\psi_i(0)(x)-\varphi_i(0)(x))\,,
\end{multline}
for the constant $L=L_\rho>0$ provided in $\rm{(C)}$.
To see it, just subtract and add the term $f_i(\w{\cdot}t,x,\phi(0)(x),\varphi(-1)(x))$ for the map $\phi\in C_\gamma$ defined by $\phi_i=\psi_i$ and  $\phi_j=\varphi_j$ if $j\not= i$, which satisfies $\varphi\leq \phi\leq \psi$, and then apply (QM) and (C).
\par
Now, let us fix $\varphi,\psi\in
C_\gamma$ with $\varphi\leq \psi$ and such that $z(t,\w,\varphi)$ and $z(t,\w,\psi)$ are defined for
$t\in[0,\beta]$, and let $\rho>0$ be such that $\sup\{\|z(t,\w,\varphi)(x)\|,\,\|z(t,\w,\psi)(x)\|\mid t\in[-1,\beta],\, x\in \bar U\}< \rho$. Then take $L=L_\rho$ the constant given in (C), which obviously depends on $\w,\,\varphi,\,\psi$ and $\beta$.
\par
As a first step, we consider the particular case when $\varphi,\psi\in C^\theta([-1,0],E)$, and $\varphi(0),\,\psi(0)\in C^{2\theta}(\bar U, \R^n)$ in the Neumann or Robin cases.  Then, either Theorem~\ref{teor-time regularity-Neumann} or Theorem~\ref{teor-time regularity-Dirichlet} applies to get that the mild solutions $z(t,\w,\varphi)$ and $z(t,\w,\psi)$ are classical solutions on $[0,\beta]$, so that for each fixed $i=1,\ldots,n$ we can consider the map on $[0,\beta]$, with values in $D(A_i)$ for $t>0$, defined by
$v_i(t)=e^{L t}\,(z_i(t,\w,\psi)-z_i(t,\w,\varphi))$.
Then, for $t>0$, and for $F:\W\times C_\gamma\to E$ defined in~\eqref{F}, we have that
\begin{equation*}
v_i'(t)= L v_i(t)+ A_i
v_i(t)+e^{L t} (F_i(\w{\cdot}t,z_t(\w,\psi))-
F_i(\w{\cdot}t,z_t(\w,\varphi))).
\end{equation*}
Now, for any $t\in(0,b]$, since $z_t(\w,\varphi)\leq z_t(\w,\psi)$  by (i), and by the choice of $\rho$,~\eqref{lemma 1.1} applies and we can write $v_i'(t)\geq L v_i(t)+A_i v_i(t)-L v_i(t)=
A_i v_i(t)$,
so that $g_i(t)=v_i'(t)-A_i v_i(t)\geq 0$. Now,  $(e^{tA_i})_{t\geq 0}$ is
a positive semigroup of operators in $C(\bar U)$ in the Neumann or Robin cases, and in $C_0(\bar U)$ in the Dirichlet case  (for instance, see Smith \cite{smit}).  Besides, in the Dirichlet case, $\overline{ D(A_i)}=C_0(\bar U)$ and {\rm (DM)} is assumed, so that $g_i(t)=L v_i(t)+e^{L t} (F_i(\w{\cdot}t,z_t(\w,\psi))-
F_i(\w{\cdot}t,z_t(\w,\varphi)))\in C_0(\bar U)$ for $t>0$.
Finally, since $v_i'(t)=A_i v_i(t)+g_i(t)$, we can write
\begin{equation*}
v_i(t)=e^{tA_i}\,v_i(0)+\int_0^t e^{(t-s)A_i}\,g_i(s)\,ds
\geq e^{tA_i}\,v_i(0) = e^{tA_i}(\psi_i(0)-\varphi_i(0)) \,,
\end{equation*}
from where the searched inequality immediately follows.
\par
In the general case, note that  the set of H\"{o}lder continuous maps $C^\theta([-1,0],E^\gamma)$, with $\varphi(0)\in C^{2\theta}(\bar U, \R^n)$ in the Neumann or Robin cases, is dense in $C_\gamma$, and in the Dirichlet case $C^\theta([-1,0],E^\alpha)\subset C^\theta([-1,0],E)$. Then, for $\varphi,\psi\in C_\gamma$ as before, we can take sequences $\{\varphi_n\},\{\psi_n\}$ as in the first step
with $\varphi_n\to\varphi$ and  $\psi_n\to\psi$,
$\varphi_n\leq \varphi\leq \psi\leq \psi_n$ for any $n\geq 1$ and
such that
$\|z(t,\w,\varphi_n)(x)\|,\,\|z(t,\w,\psi_n)(x)\|\leq
\rho$ for any $t\in[0,\beta]$, $x\in \bar U$ and $n\geq 1$. Then, the proof is finished  by applying the
first step to the pairs $\varphi_n,\psi_n$ and taking limits as
$n\to\infty$.
\end{proof}
 Note that the standard parabolic maximum principle implies that $e^{t A_i}$ is strong\-ly positive for $t>0$, i.e., if $z_i>0$, then $e^{t A_i}z_i\gg 0$. Then, in the situation of the previous result, if $\varphi_i(0)<\psi_i(0)$ for some $i$, it is $z_i(t,\w,\varphi)\ll z_i(t,\w,\psi)$ for $t>0$.
%%%%%%%%%%%%%%%%%%%%%%%%%%%%%%%%%%%%%%%%%%%%%%%%%%%%%%%%%%%%%%%%%%%%%%%%%%%%%%%%%
%%%%%%%%%%%%%%%%%%%%%%%%%%%%%%%%%%%%%%%%%%%%%%%%%%%%%%%%%%%%%%%%%%%%%%%%%%%%%%%%%%%%%%%%
\section{The linearized semiflow and Lyapunov exponents}\label{sec-linearized sem} \noindent
In this section we build the linearized semiflow under regularity conditions in the problems. Besides, when the semiflow is also monotone, we present the concept of a continuous separation of type~II and of the related principal spectrum, and show how the latter can be calculated in terms of some Lyapunov exponents.
\par
From now on, we use the unified notation introduced in Section~\ref{sec-monotonos} to include any of the boundary conditions, but whenever it is convenient to make a distinction, we will write $C=C([-1,0],E)$ with sup-norm $\|\,{\cdot}\,\|_{C}$ in the Neumann and Robin cases, and  $C_\alpha=C([-1,0],E^\alpha)$  with sup-norm $\|\,{\cdot}\,\|_{C_\alpha}$ in the Dirichlet case.
\par
In the case of Neumann boundary conditions, the next result can be found in Novo et al.~\cite{nonuobsa}, and it can be  trivially extended to the case of Robin boundary conditions. The proof is inspired in the proof of Theorem~3.5 in Novo et al.~\cite{nonuobsa}.
\begin{teor}\label{teor-linearized sk}
Consider the family of IBV problems with delay~\eqref{family}, $\w\in\W$ and assume that $f:\W\times\bar U\times \R^n\times \R^n\to\R^n$ is continuous and  of class $C^1$ in the $y$ and $\wit y$ variables. Then, the skew-product semiflow generated by mild solutions on $\W\times C_\gamma$, $\tau(t,\w,\varphi)=(\w{\cdot}t,z_t(\w,\varphi))$  is  of class $C^1$ with respect to $\varphi$. Furthermore, for each  $\psi\in C_\gamma$, $D_{\!\varphi} z_t(\w,\varphi)\,\psi =
v_t(\w,\varphi,\psi)$ for the mild solution $v(t,\w,\varphi,\psi)$
of the associated variational retarded ACP along the semiorbit of
$(\w,\varphi)$,
\begin{equation}\label{variational}
\left\{\begin{array}{l} v'(t)=Av(t)+D_{\!\varphi}
F(\w{\cdot}t,z_t(\w,\varphi))\,v_t\,, \quad t> 0\,,
\\v_0=\psi\in C_\gamma\,,\end{array}\right.
\end{equation}
which is defined for $t$ in $[0,b)$, the maximal interval of definition of $z(t,\w,\varphi)$.
\end{teor}
\begin{proof}
We write the proof for the case of Dirichlet boundary conditions. Recall that $F:\W\times C_\alpha\to E$ is defined in~\eqref{F} and that $z(t,\w,\varphi)$ is a mild solution of the retarded ACP~\eqref{ACPdelay}. In this case with fixed delay,
\begin{multline}\label{derivada}
[D_{\!\varphi}
F(\w{\cdot}t,z_t(\w,\varphi))\,v_t](x)=D_{y} f(\w{\cdot}t,x,z(t,\w,\varphi)(x),z(t-1,\w,\varphi)(x))\,v(t)(x)\\
+D_{\wit y} f(\w{\cdot}t,x,z(t,\w,\varphi)(x),z(t-1,\w,\varphi)(x))\,v(t-1)(x)\,,\quad x\in \bar U.
\end{multline}
\par
By the $C^1$ character of $f(\w,x,y,\wit y)$ in $(y,\wit y)$, we can argue as in the previous sections to get the existence of a unique mild solution of~\eqref{variational}, denoted by $v(t)=v(t,\w,\varphi,\psi)$. By linearity of the problem, $v$ exists in the large, i.e.,
\begin{equation}\label{mild v}
v(t)=e^{tA}\,\psi(0) +\int_0^t e^{(t-s)A}\,D_{\!\varphi}
F(\w{\cdot}s,z_s(\w,\varphi))\,v_s
\,ds\,,\quad\hbox{for any }\, t\in [0,b)\,.
\end{equation}
\par
Let us fix a $t>0$, and let us
first check that for $\w\in \W$ and $\varphi,\,\psi\in
C_\alpha$, $D_{\!\varphi} z_t(\w,\varphi)\,\psi$ exists, provided that $z_t(\w,\varphi)$ exists, and $D_{\!\varphi} z_t(\w,\varphi)\,\psi =
v_t(\w,\varphi,\psi)$, and second that the map $\Om\times
C_\alpha\to \mathcal L(C_\alpha)$, $(\w,\varphi) \mapsto D_{\!\varphi}
z_t(\w,\varphi)$ is continuous.
\par
First of all, note that fixed $t>0$ and $(\w,\varphi)\in \W\times C_\alpha$ such that $z_t(\w,\varphi)$ exists, and given $\psi\in C_\alpha$, the
solution $z(\,\cdot\,,\w,\varphi+\varepsilon \,\psi)$
of~\eqref{ACPdelay} with initial data $z_0=\varphi+\varepsilon
\,\psi$ is also defined on $[0,t]$, provided that $|\varepsilon|\leq
\varepsilon_0$ for a sufficiently small $\varepsilon_0$. We want to
prove that there exists the limit
\[
\lim_{\varepsilon\to 0} \frac{z_t(\w,\varphi+\varepsilon
\,\psi)-z_t(\w,\varphi)}{\varepsilon}= v_t(\w,\varphi,\psi)\,.
\]
For convenience, we will get to this by  proving that
$\lim_{\varepsilon\to 0} h^\varepsilon(t)=0$ for the map
\[
h^\varepsilon(s)= \frac{1}{\varepsilon}
 \sup_{r\in [0,s]} \| z(r,\w,\varphi+\varepsilon
\,\psi)-z(r,\w,\varphi)-\varepsilon\, v(r,\w,\varphi,\psi)\|_\alpha\,,\quad s\in [0,t]\,.
\]
Let us call $g^\varepsilon(r)=z(r,\w,\varphi+\varepsilon
\,\psi)-z(r,\w,\varphi)-\varepsilon\, v(r,\w,\varphi,\psi)$ for $r\in[0,s]$ and recall that $\|z\|_\alpha= \|(-A_{p})^{\alpha}\,z\|_p$, for any $z\in E^\alpha$.
\par
Having in mind~\eqref{variation constants} and~\eqref{mild v} we
write for $r\in [0,s]$,
\begin{multline*}
g^\varepsilon(r)=\displaystyle\int_0^{r} e^{(r-l)A}\,\big[
F(\w{\cdot}l,z_l(\w,\varphi+\varepsilon \,\psi))-
F(\w{\cdot}l,z_l(\w,\varphi))\\ -\varepsilon D_{\!\varphi}
F(\w{\cdot}l,z_l(\w,\varphi))\,v_l(\w,\varphi,\psi)\big]\,dl\,,
\end{multline*}
and the term $F(\w{\cdot}l,z_l(\w,\varphi+\varepsilon \,\psi))-
F(\w{\cdot}l,z_l(\w,\varphi))$ can  be written, applying the mean value
theorem to $F$, as
\[
\int_0^1 D_{\!\varphi}
F(\w{\cdot}l,\lambda\,z_l(\w,\varphi+\varepsilon \,\psi)+
(1-\lambda)\,z_l(\w,\varphi))(z_l(\w,\varphi+\varepsilon
\,\psi)-z_l(\w,\varphi))\,d\lambda\,.
\]
Consequently, for any $r\in [0,s]$ we can write
\begin{multline*}
g^\varepsilon(r)= \displaystyle\int_0^{r} e^{(r-l)A}\left(\int_0^1
D_{\!\varphi}  F(\w{\cdot}l,\lambda\,z_l(\w,\varphi+\varepsilon
\,\psi)+ (1-\lambda)\,z_l(\w,\varphi))\,g_l^\varepsilon\,d\lambda
\right)dl
\\ + \varepsilon \displaystyle\int_0^{r} e^{(r-l)A}\left( \int_0^1 \big[
D_{\!\varphi} F(\w{\cdot}l,\lambda\,z_l(\w,\varphi+\varepsilon
\,\psi)+ (1-\lambda)\,z_l(\w,\varphi)) \right.\\ -
D_{\!\varphi}
F(\w{\cdot}l,z_l(\w,\varphi))\big]\,v_l(\w,\varphi,\psi)\,d\lambda
\Big)\,dl \,.
\end{multline*}
From this, taking into account that:
\par\smallskip
- there exists an $M_\alpha'>0$ such that
\begin{equation}\label{cota}
\|(-A_{p})^\alpha\,e^{rA}\,z\|_p\leq M_\alpha'\,r^{-\alpha}\,\|z\| \quad\text{for any}\; z\in E\;\text{and}\; r>0,
\end{equation}
because of~\eqref{estimate} and the continuous embedding $E\hookrightarrow L^p(U,\R^n)$;

-  $\|v_l(\w,\varphi,\psi)\|_{C_\alpha}\leq K_1$ for some $K_1>0$ and for any $l\in[0,t]\,$;

- $\sup_{\lambda\in[0,1]} \| D_{\!\varphi}  F(\w{\cdot}l,\lambda\,z_l(\w,\varphi+\varepsilon \,\psi)+
(1-\lambda)\,z_l(\w,\varphi))\|_{\mathcal L(C_\alpha,E)}\leq K_2$ for some $K_2>0$, for
any $l\in[0,t]$ and for small enough $|\varepsilon|$, because of the
continuity of $D_{\!\varphi} F$ in $\Om\times C_\alpha$  and the
compactness of $\{z_l(\w,\varphi) \mid l\in[0,t]\}$ for the norm in
$C_\alpha$;

- $\lim_{\varepsilon\to 0}\alpha^\varepsilon(s)=0$ uniformly for $s\in[0,t]$, where  \par
\noindent${\displaystyle \alpha^\varepsilon(s)=M_\alpha'\,K_1\!\!\!\sup_{r\in[0,s]}\!\int_0^r
(r-l)^{-\alpha}\left(\int_0^1 \!\!\| D_{\!\varphi}
F(\w{\cdot}l,\lambda\,z_l(\w,\varphi+\varepsilon \,\psi)+
(1-\lambda)\,z_l(\w,\varphi)) \right.}$\\
${\displaystyle ~\hspace{5.5cm}- D_{\!\varphi} F(\w{\cdot}l,z_l(\w,\varphi))\|_{\mathcal L(C_\alpha,E)}\,d\lambda\Big)dl\,;}$

- the map $h^\varepsilon(l)$ is nondecreasing for $l\in [0,t]$, hence, for $s\in[0,t]$,
\[
\sup_{r\in[0,s]}\int_0^r M_\alpha'\,(r-l)^{-\alpha}
K_2 \, h^\varepsilon(l)\,dl=\int_0^s M_\alpha'\,(s-l)^{-\alpha}
K_2 \, h^\varepsilon(l)\,dl\,;
\]
\par\smallskip\noindent we obtain that for any $s\in [0,t]$,
\[
h^\varepsilon(s)=\frac{1}{\varepsilon}\,\sup_{r\in[0,s]}
\|g^\varepsilon(r)\|_\alpha\leq \alpha^\varepsilon(s)+ \int_0^s M_\alpha'\,(s-l)^{-\alpha}
K_2 \, h^\varepsilon(l)\,dl\,,
\]
and applying the generalized Gronwall's lemma, we get  that for any $s\in [0,t]$,
\[
h^\varepsilon(s)\leq \alpha^\varepsilon(s)+ \theta \int_0^s H(\theta(s-l))\,\alpha^\varepsilon(l)\,dl\,,
\]
where the constant $\theta$ depends on the constants $M_\alpha'\,K_2$ and on  $1-\alpha$, and the map $H(s)$ behaves like  $\frac{s^{-\alpha}}{\Gamma(1-\alpha)}$ as $s\to 0^+$  (see Lemma 7.1.1 in Henry~\cite{henr} for more details). From here, we can deduce that $\lim_{\varepsilon\to 0}h^\varepsilon(t)=0$, as we wanted to see.
\par
To finish the proof, let us fix a $t>0$ and let us check the continuity of the
map $\Om\times C_\alpha\to  \mathcal L(C_\alpha)$, $(\w,\varphi) \mapsto
D_{\!\varphi} z_t(\w,\varphi)$. So, let us take
$\{(\w_n,\varphi_n)\}_{n\geq 1}\subset \Om\times C_\alpha$ with
$(\w_n,\varphi_n)\to (\w,\varphi)$ and let us see that
\begin{align*}
\|D_{\!\varphi} &z_t(\w_n,\varphi_n)-D_{\!\varphi}
z_t(\w,\varphi)\|_{\mathcal L(C_\alpha)}=\sup_{\|\psi\|\leq
1}\|v_t(\w_n,\varphi_n,\psi)- v_t(\w,\varphi,\psi)\|_{C_\alpha}\\
&\leq \sup_{\|\psi\|\leq 1}\sup_{s\in
[0,t]}\|v(s,\w_n,\varphi_n,\psi)- v(s,\w,\varphi,\psi)\|_\alpha\to 0\quad\hbox{as}\; n\to\infty\,.
\end{align*}
The general arguments are similar to the ones used before. Using~\eqref{mild v},
we first apply  the generalized Gronwall's
inequality to prove that $\|v_s(\w,\varphi,\psi)\|_{C_\alpha}$ is
uniformly bounded for $s\in[0,t]$ and $\|\psi\|\leq
1$. Note that~\eqref{sup finito} is needed at this point. Then, again using~\eqref{mild v}  for
$v(s,\w_n,\varphi_n,\psi)- v(s,\w,\varphi,\psi)$, a further
application of the generalized Gronwall's inequality, together with~\eqref{cota} and the facts that:
\par\smallskip
- $\des\sup_{s\in[0,t]}\des\sup_{n\geq 1}\|D_{\!\varphi}
F(\w_n{\cdot}s,z_s(\w_n,\varphi_n))\|_{\mathcal L(C_\alpha,E)}<\infty\,;$

-$\des\lim_{n\to\infty}\des\sup_{s\in[0,t]}\|D_{\!\varphi}
F(\w_n{\cdot}s,z_s(\w_n,\varphi_n))-D_{\!\varphi}
F(\w{\cdot}s,z_s(\w,\varphi))\|_{\mathcal L(C_\alpha,E)}= 0\,;$

-  $l\in [0,t]\mapsto \des\sup_{r\in[0,l]}\des\sup_{\|\psi\|\leq
1}\|v(r,\w_n,\varphi_n,\psi)- v(r,\w,\varphi,\psi)\|_\alpha$ is nondecreasing;
\par\smallskip\noindent
permits to see that the above limit is $0$. The proof is finished.
\end{proof}
In the conditions of the previous result, if there is a compact positively invariant set $K\subset \W\times C_\gamma$ for $\tau$ (e.g., if there is a bounded solution $z(t,\w,\varphi)$ and $K$ is the omega-limit set of $(\w,\varphi)$), one can build the linearized skew-product semiflow over $K$:
\begin{equation*}
 \begin{array}{cccl}
 L: & \R_+\times K \times C_\gamma& \longrightarrow & K \times C_\gamma\\
& (t,(\w,\varphi),\psi) & \mapsto &(\tau(t,\w,x),v_t(\w,\varphi,\psi))\,,
\end{array}
\end{equation*}
with $v_t(\w,\varphi,\psi)=D_{\!\varphi} z_t(\w,\varphi)\,\psi$, and  $v(t,\w,\varphi,\psi)$ is the mild solution of the variational retarded ACP~\eqref{variational} along the semiorbit of
$(\w,\varphi)$. Note that, because of boundedness of $K$, the semiflow inside  $K$ is globally defined.
\par
It is important to note that, if $K$ is $\tau$-invariant and compact, in the Dirichlet case $K$ can be equally considered with either the topology of $\W\times C_\alpha$ or of $\W\times C$.
\begin{prop}
If $K$ is a compact $\tau$-invariant subset of $\Om\times C$, then $K \subset \Omega \times C_\alpha$ and the restriction  of both topologies on $K$ agree.
\end{prop}
\begin{proof}
Since $\tau_t(K)=K$ for any $t\geq 0$ and $\tau_t:\W\times C\to \W\times C_\alpha$ is compact for $t>1$, $K$ is relatively compact in $\W\times C_\alpha$; and it is closed because the inclusion $\Om\times C_\alpha\hookrightarrow \Om\times C$ is continuous. Thus, the identity map with the two topologies
$i:(K,\W\times C_\alpha)\to (K,\W\times C)$
is a homeomorphism, as it is continuous, bijective and $(K,\W\times C_\alpha)$ is compact.
\end{proof}
Closely related to the classical concept of a continuous separation in the terms given by Pol\'{a}\v{c}ik and
Tere\v{s}\v{c}\'{a}k~\cite{pote} and Shen and Yi~\cite{shyi}, Novo et al.~\cite{noos6} introduced the concept of a continuous separation of type~II, which is the appropriate one if there is delay in the equations. We include the definition here, since it is going to be crucial in the study of persistence properties in Section~\ref{sec-uniform persistence}.
\par
When the skew-product semiflow $\tau$ is monotone and of class $C^1$ in $\varphi$, we say that a compact, positively invariant set $K\subset \W\times C_\gamma$ admits a {\em continuous separation of type~\/}II if there are families of
subspaces $\{X_1(\w,\varphi)\}_{(\w,\varphi)\in K}$ and
$\{X_2(\w,\varphi)\}_{(\w,\varphi)\in K}$ of $C_\gamma$ satisfying the following
properties.
\begin{itemize}
\item[(S1)$\;$] $C_\gamma=X_1(\w,\varphi)\oplus X_2(\w,\varphi)$  and $X_1(\w,\varphi)$,
$X_2(\w,\varphi)$ vary
    continuously in $K$;
 \item[(S2)$\;$] $X_1(\w,\varphi)=\spa\{ \psi(\w,\varphi)\}$, with $\psi(\w,\varphi)\gg 0$ and
     $\|\psi(\w,\varphi)\|_{C_\gamma}=1$ for any $(\w,\varphi)\in K$;
\item[(S3)'] there exists a $t_0>0$ such that if for some $(\w,\varphi)\in
K$ there is a
    $\phi\in X_2(\w,\varphi)$ with $\phi>0$, then $D_{\!\varphi} z_t(\w,\varphi)\,\phi=0$ for any $t\geq t_0$;
\item[(S4)$\;$] for any $t>0$,  $(\w,\varphi)\in K$,
\begin{align*}
D_{\!\varphi} z_t(\w,\varphi)\,X_1(\w,\varphi)&= X_1(\tau(t,\w,\varphi))\,,\\
D_{\!\varphi} z_t(\w,\varphi)\,X_2(\w,\varphi)&\subset X_2(\tau(t,\w,\varphi))\,;
\end{align*}
\item[(S5)$\;$] there are $M>0$, $\delta>0$ such that for any
$(\w,\varphi)\in K$, $\phi\in
    X_2(\w,\varphi)$ with $\|\phi\|_{C_\gamma}=1$ and $t>0$,
\begin{equation*}
\|D_{\!\varphi} z_t(\w,\varphi)\,\phi\|_{C_\gamma}\leq M \,e^{-\delta
t}\,\|D_{\!\varphi} z_t(\w,\varphi)\,\psi(\w,\varphi)\|_{C_\gamma}\,.
\end{equation*}
\end{itemize}
The precise meaning of  the continuous variation expressed in  (S1) has been explained in Obaya and Sanz~\cite{obsa}.
\par
For convenience, we also recall some definitions of Lyapunov exponents. The standard definition of  superior and inferior Lyapunov exponents at $\infty$ of  each $(\w,\varphi,\psi)\in K\times C_\gamma$ is as follows (for instance, see Sacker and Sell~\cite{sase}):
\[
\lambda_i(\w,\varphi,\psi)=\liminf_{t\to\infty} \frac{\log\|v_t(\w,\varphi,\psi)\|_{C_\gamma}}{t}\,,\;
\lambda_s(\w,\varphi,\psi)=\limsup_{t\to\infty} \frac{\log\|v_t(\w,\varphi,\psi)\|_{C_\gamma}}{t}\,;
\]
the Lyapunov exponents of each $(\w,\varphi)\in K$ are defined by
\[
\lambda_i(\w,\varphi)=\liminf_{t\to\infty} \frac{\log\|D_{\!\varphi} z_t(\w,\varphi)\|_{\mathcal{L}(C_\gamma)}}{t},\,
\lambda_s(\w,\varphi)=\limsup_{t\to\infty} \frac{\log\|D_{\!\varphi} z_t(\w,\varphi)\|_{\mathcal{L}(C_\gamma)}}{t};
\]
and the lower and upper Lyapunov exponents of $K$ are respectively the numbers:
$\alpha_K=\inf_{(\w,\varphi)\in K} \lambda_i(\w,\varphi)$ and $\lambda_K=\sup_{(\w,\varphi)\in K} \lambda_s(\w,\varphi)$.
\par
When the linearized semiflow $L$ is monotone and $K$ is a minimal set with a flow extension and a continuous separation of type~II, these exponents play a fundamental role in the determination of the principal spectrum $\Sigma_p$ (see Mierczy{\'n}ski and Shen~\cite{mish}), that is, the Sacker-Sell spectrum (see~\cite{sase,sase94}) of the restriction of $L$ to the one-dimensional invariant subbundle
\begin{equation*}%\label{one-dim-family}
\displaystyle\bigcup_{(\w,\varphi)\in K} \{(\w,\varphi)\} \times X_1(\w,\varphi)\,.
\end{equation*}
More precisely, $\Sigma_p=[\alpha_K,\lambda_K]$  and besides, if $X_1(\w,\varphi)=\spa\{\psi\}$ for the vector $\psi=\psi(\w,\varphi)\gg 0$ in (S2), then
$\lambda_i(\w,\varphi)=\lambda_i(\w,\varphi,\psi)$ and $\lambda_s(\w,\varphi)=\lambda_s(\w,\varphi,\psi)$
(see Proposition~4.4 in Novo et al.~\cite{noos7} for the result in an abstract setting). Since principal spectrums are going to be the dynamical objects in order to determine the persistence of the systems in Section \ref{sec-uniform persistence}, it is good to know that in the Dirichlet case the Lyapunov exponents can be calculated with the sup-norm in $C=C([-1,0],E)$, which is much easier to deal with numerically than the sup-norm in $C_\alpha$.
\begin{prop}\label{prop-Lyapunov exponents}
Assume that the map $f$ in~\eqref{family} is continuous and  of class $C^1$ in the $y$ and $\wit y$ variables. Let $K\subset \W\times C_\gamma$ be a compact positively invariant set and consider the linearized semiflow $L$ over $K$. Then, in the case of Dirichlet boundary conditions, for any $(\w,\varphi)\in K$ and $\psi \in C_\alpha$ one can calculate:
\[
\lambda_i(\w,\varphi,\psi)=\liminf_{t\to\infty} \frac{\log\|v_t(\w,\varphi,\psi)\|_{C}}{t}\,,\;
\lambda_s(\w,\varphi,\psi)=\limsup_{t\to\infty} \frac{\log\|v_t(\w,\varphi,\psi)\|_{C}}{t}\,.
\]
In particular, if $\lambda_i(\w,\varphi,\psi)=\lambda_s(\w,\varphi,\psi)$, then
\[
\lambda(\w,\varphi,\psi)=\lim_{t\to\infty} \frac{\log\|v_t(\w,\varphi,\psi)\|_{C}}{t}\,.
\]
\end{prop}
\begin{proof}
Let us omit the dependence of $v_t$ on $(\w,\varphi,\psi)$ to simplify the writing, and set $\tilde\lambda_s=\limsup_{t\to\infty} \frac{\log\|v_t\|_{C}}{t}$. Since $C_\alpha\hookrightarrow C$, it is clear that $\tilde\lambda_s\leq \lambda_s$. To see that also $\lambda_s\leq \tilde\lambda_s$, let us take a sequence $t_n\uparrow \infty$ such that $\lambda_s=\lim_{n\to\infty} \frac{\log\|v_{t_n}\|_{C_\alpha}}{t_n}$. Since for each $n\geq 1$ there exists a $t_n^1\in [t_n-1,t_n]$ such that $\|v_{t_n}\|_{C_\alpha}=\|v(t_n^1)\|_{\alpha}$, we have that $\lambda_s=\lim_{n\to\infty} \frac{\log\|v(t_n^1)\|_{\alpha}}{t_n^1}$. Now,  for the map $F:\W\times C_\alpha\to E$ defined in~\eqref{F}, we can write by the variation of constants formula~\eqref{variation constants},
\[
v(t_n^1)=e^A\,v(t_n^1-1) +\int_0^1 e^{(1-s)A}\,D_{\!\varphi}F(\w{\cdot}(t_n^1-1+s),z_{t_n^1-1+s}(\w,\varphi))\,v_{t_n^1-1+s}\,ds\,.
\]
Now, note that we can also consider $F$ as defined on $\W\times C$ with values in $E$. Then, taking $M=\sup\{\|D_{\!\varphi}F(\tilde \w, \tilde \varphi)\|_{\mathcal{L}(C,E)}\mid (\tilde \w, \tilde \varphi)\in K\}<\infty$, %(see~\eqref{derivada} and note that $K$ is compact)
we can apply~\eqref{cota} to get
\[
\|v(t_n^1)\|_{\alpha}\leq M'_\alpha\, \|v(t_n^1-1)\|+\int_0^1 M'_\alpha\, (1-s)^{-\alpha}\, M\, \|v_{t_n^1-1+s}\|_C\,ds\,.
\]
As before, for each $n\geq 1$, there exists a $t_n^2\in [t_n^1-2,t_n^1]$ such that $\|v_{t_n^1-1+s}\|_C\leq \|v(t_n^2)\|$ for any $s\in [0,1]$, and in particular $\|v(t_n^1-1)\|\leq \|v(t_n^2)\|$. Then,
\[
\|v(t_n^1)\|_{\alpha}\leq M'_\alpha\, \|v(t_n^2)\|+\int_0^1 M'_\alpha\, (1-s)^{-\alpha}\, M\, \|v(t_n^2)\|\,ds=\big(1+\frac{M}{1-\alpha}\big) M'_\alpha\, \|v(t_n^2)\| \,,
\]
for any $n\geq 1$. Since $\|v(t_n^2)\| \leq \|v_{t_n^2}\|_C $, we can easily conclude that
\[
\lambda_s=\lim_{n\to\infty} \frac{\log\|v(t_n^1)\|_{\alpha}}{t_n^1}\leq \lim_{n\to\infty} \frac{\log\|v_{t_n^2}\|_{C}}{t_n^2} \leq \limsup_{t\to\infty} \frac{\log\|v_t\|_{C}}{t}= \tilde\lambda_s\,.
\]
\par
Let us now deal with $\tilde\lambda_i=\liminf_{t\to\infty} \frac{\log\|v_t\|_{C}}{t}$. Once more the inequality  $\tilde\lambda_i\leq \lambda_i$ is clear, so that it remains to prove that $\lambda_i\leq \tilde\lambda_i$. This time we take a sequence $t_n\uparrow\infty$ such that $\tilde\lambda_i=\lim_{n\to\infty} \frac{\log\|v_{t_n}\|_{C}}{t_n}$. Now, arguing as in the first paragraph, associated with the sequence $\{t_n+2\}_{n\geq 1}$ we can find a sequence $\{t_n^2\}_{n\geq 1}$ with $t_n^2\in [t_n-1,t_n+2]$ such that $\|v_{t_n+2}\|_{C_\alpha}\leq c\,\|v(t_n^2)\|$ for $c=(1+\frac{M}{1-\alpha}) M'_\alpha>0$ and for any $n\geq 1$. Note that if we prove that $\|v(t_n^2)\|\leq \tilde c\,\|v_{t_n}\|_C$ for every $n\geq 1$, for a certain $\tilde c>0$,  we are done, since then:
\[
\lambda_i=\liminf_{t\to\infty} \frac{\log\|v_t\|_{C_\alpha}}{t}\leq \lim_{n\to\infty}\frac{\log\|v_{t_n+2}\|_{C_\alpha}}{t_n+2}
\leq \lim_{n\to\infty} \frac{\log\|v_{t_n}\|_{C}}{t_n}=\tilde \lambda_i\,.
\]
For that, once more we use the variation of constants formula to write, for $r\in [0,2]$,
\[
v(t_n+r)=e^{rA}\,v(t_n) +\int_0^r e^{(r-l)A}\,D_{\!\varphi}F(\w{\cdot}(t_n+l),z_{t_n+l}(\w,\varphi))\,v_{t_n+l}\,dl\,.
\]
Then, consider the map $h_n(s)=\des\sup_{r\in [-1,s]}\|v(t_n+r)\|$ defined for $s\in [0,2]$. Note that if for some $s\in [0,2]$, $h_n(s)=\|v(t_n+r_0)\|$ for some $-1\leq r_0\leq 0$, then $h_n(s)\leq \|v_{t_n}\|_C$. Else, $h_n(s)=\des\sup_{r\in [0,s]}\|v(t_n+r)\|$ and we can bound
\[
h_n(s)\leq M_0\,\|v_{t_n}\|_C+ \int_0^s M_0\,M\,h_n(l)\,dl\,
\]
for the constants $M_0=\max\{1,\sup_{s\in [0,2]}\|e^{sA}\|_{\mathcal{L}(E)}\}$ and $M$ the same as before, so that this inequality holds for any $s\in [0,2]$. Applying of the Gronwall's lemma, we obtain $h_n(s)\leq \tilde c \,\|v_{t_n}\|_C$ for an appropriate $\tilde c>0$ independent of $n\geq 1$, for any $s\in [0,2]$. In particular $\|v(t_n^2)\|\leq h_n(2)\leq \tilde c\,\|v_{t_n}\|_C$ for every $n\geq 1$. The proof is finished.
\end{proof}
%%%%%%%%%%%%%%%%%%%%%%%%%%%%%%%%%%%%%%%%%%%%%%%%%%%%%%%%%%%%%%%%%%%%%%%%%%%%
%%%%%%%%%%%%%%%%%%%%%%%%%%%%%%%%%%%%%%%%%%%%%%%%%%%%%%%%%%%%%%%%%%%%%%%%%%%%
%%%%%%%%%%%%%%%%%%%%%%%%%%%%%%%%%%%%%%%%%%%%%%%%%%%%%%%%%%%%%%%%%%%%%%%%%%%
\section{Persistence for quasimonotone systems of parabolic PFDEs}\label{sec-uniform persistence}\noindent
In this section  the properties of uniform and strict persistence are studied for quasimonotone and regular parabolic problems of type~\eqref{family}, $\w\in\W$. More precisely, we  assume the following conditions on $f$:
\begin{itemize}
\item[(C1)] $f(\w,x,y,\wit y)$ is continuous and of class  $C^1$ in $(y,\wit y)$.
\item[(C2)] The maps $D_y f(\w{\cdot}t,x,y,\wit y)$ and
$D_{\wit y} f(\w{\cdot}t,x,y,\wit y)$ are Lipschitz in $(y,\wit y)$ in bounded sets, uniformly for $\w\in \Om$ and $x\in\bar U$.
\item[(C3)] $f(\w{\cdot}t,x,y,\wit y)$ as well as the maps $D_yf(\w{\cdot}t,x,y,\wit y)$ and
$D_{\wit y} f(\w{\cdot}t,x,y,\wit y)$ satisfy conditions $(C^{\theta}(t))$ and $(C^{2\theta}(x))$, for some $\theta\in(0,1/2)$.
\item[(C4)] Quasimonotone condition: for any
$(\w,x,y,\wit y)\in \W\times \bar U\times\R^n\times\R^n$,
\[
\frac{\partial f_i}{\partial y_j}(\w,x,y,\wit y) \geq 0 \;\,
 \text{ for } i\not= j \;\, \text{ and }\;  \frac{\partial f_i}{\partial \wit y_j}(\w,x,y,\wit y)\ge 0\,\;\text{ for any}\;\, i, j\,.
\]
\end{itemize}
\par
As proved in Theorem \ref{teor-linearized sk}, with (C1) the skew-product semiflow $\tau(t,\w,\varphi)$ is of class $C^1$ in $\varphi$. Condition $(C^{2\theta}(x))$ in (C3) is required so that the solutions of the IBV problems with delay, as well as those of the linearized problems, are smooth enough in order to apply the classical parabolic maximum or minimum principles; see Theorems~\ref{teor-space regularity-Neumann} and~\ref{teor-space regularity-Dirich}. Finally, note that (C4) is the usual way to write the quasimonotone condition (QM) under regularity assumptions.
\par
First of all, by linearizing the problems, in the Dirichlet case we can now establish the monotonicity of the skew-product semiflow removing condition (DM) in Proposition~\ref{prop-monotone}. Recall that $C_\alpha=C([-1,0],E^\alpha)$.
\begin{prop}\label{prop-strong monotonicity-Dirichlet}
Consider the family of parabolic problems with delay~\eqref{family}, $\w\in\W$ with Dirichlet boundary conditions and assume that  $f$ satisfies $\rm{(C1)}$-$\rm{(C4)}$.
Then:
\begin{itemize}
\item[(i)] The induced skew-product semiflow on $\W\times C_\alpha$
is monotone, that is, if $\varphi,\psi\in
C_\alpha$ with $\varphi\leq \psi$, then $z_t(\w,\varphi)\leq
z_t(\w,\psi)$ for any $\w\in\W$ and any $t\geq 0$ where both terms
are defined.
\item[(ii)] Given $\w\in\W$ and $\varphi,\psi\in
C_\alpha$ with $\varphi\leq \psi$ such that $z(t,\w,\varphi)$ and $z(t,\w,\psi)$ are defined for
$t\in[0,\beta]$ for some $\beta>0$, there exists an $L=L(\w,\varphi,\psi,\beta)>0$ such that
for each $i=1,\ldots,n$, and for each $t\in[0,\beta]$,
\[
z_i(t,\w,\psi)-z_i(t,\w,\varphi)\geq e^{-L
t}\,e^{t A_i}\,(\psi_i(0)-\varphi_i(0))\,.
\]
\end{itemize}
\end{prop}
\begin{proof}
Note that with any boundary conditions, by the regularity assumptions on $f$ we can consider the linearized IBV problem of~\eqref{family} along the semiorbit of each  fixed $(\w,\varphi)\in \W\times C_\gamma$,
\begin{equation}\label{linear family}
\left\{\begin{array}{l} \des\frac{\partial u}{\partial t}=
 D\Delta u+g(\tau(t,\w,\varphi),x,u(t,x),u(t-1,x))\,,\; t\in (0,\beta]\,,\;\,x\in \bar U,\\[.2cm]
\bar\alpha(x)\,u(t,x)+\delta\,\des\frac{\partial u}{\partial n}(t,x) =0\,,\quad  t\in (0,\beta]\,,\;\,x\in \partial U,\\[.2cm]
u(s,x)=\psi(s,x)\,, \quad s\in [-1,0]\,,\;\,x\in \bar U,
\end{array}\right.
\end{equation}
provided that the mild solution $z(t,\w,\varphi)$ is defined on the interval $[0,\beta]$, where the map $g:(\W\times C_\gamma)\times\bar U\times \R^n\times\R^n\to\R^n$, linear in $(u,v)$, is defined by (see~\eqref{derivada})
%\begin{multline}%\label{g}
\[g(\w,\varphi,x,u,v)=D_{y} f(\w,x,\varphi(0)(x),\varphi(-1)(x))\,u
+D_{\wit y} f(\w,x,\varphi(0)(x),\varphi(-1)(x))\,v\,.\]
%\end{multline}
\par
Under assumptions (C1)-(C4) on $f$, it is easy to check that $g$ satisfies all the conditions in order to apply  Proposition~\ref{prop-monotone} to each linearized problem along the orbit of $(\w,\varphi)$ if $\varphi\in C^\theta([-1,0],E)$. Let us now restrict to the Dirichlet case.
\par
Arguing as in the proof of Proposition~\ref{prop-monotone}~(ii), we just need to consider $\w\in\W$ and $\varphi,\psi\in
C_\alpha$ with $\varphi\leq \psi$ such that $\varphi,\,\psi\in C^\theta([-1,0],E)$, and  $z(t,\w,\varphi)$ and $z(t,\w,\psi)$ are defined for
$t\in[0,\beta]$ for some $\beta>0$: in the general case, we can approximate $\varphi$ and $\psi$ by $\theta$-H\"{o}lder continuous maps.
Besides, we can assume without loss of generality that also $z(t,\w,\lambda \psi + (1-\lambda)\varphi)$ is defined for
$t\in[0,\beta]$  for every $\lambda\in (0,1)$. Then, thanks to Theorem~\ref{teor-linearized sk} we can write for any $t\in (0,\beta]$,
\begin{equation}\label{mean value}
z_t(\w,\psi)-z_t(\w,\varphi)=\int_0^1 D_{\!\varphi} z_t(\w,\lambda\, \psi + (1-\lambda)\,\varphi)\,(\psi-\varphi)\,d\lambda
\end{equation}
where $D_{\!\varphi} z_t(\w,\lambda \psi + (1-\lambda)\varphi)\,(\psi-\varphi)=v_t(\w,\lambda \psi + (1-\lambda)\varphi,\psi-\varphi)$ for the mild solution $v$ of the variational retarded ACP along the semiorbit of $(\w,\lambda \psi + (1-\lambda)\varphi)$ with initial condition $\psi-\varphi$ (see~\eqref{variational}), which is just the ACP built from the linearized IBV problem, to which Proposition~\ref{prop-monotone} applies. Therefore (i) immediately follows, since   $D_{\!\varphi} z_t(\w,\lambda\, \psi + (1-\lambda)\,\varphi)\,(\psi-\varphi)\geq 0$ for any $\lambda\in [0,1]$.
\par
Now, to see (ii)  just write
\begin{equation}\label{mean value componente}
z_i(t,\w,\psi)-z_i(t,\w,\varphi)=\int_0^1 v_i(t,\w,\lambda\, \psi + (1-\lambda)\,\varphi,\psi-\varphi)\,d\lambda\,,
\end{equation}
recall that $v$ is linear with respect to the initial value, and apply
Proposition~\ref{prop-monotone}~(ii) to the linearized problem for each $\lambda\in [0,1]$.
\end{proof}
In the next result conditions are given to provide the existence of a continuous separation of type~II over a minimal set $K\subset \W\times C_\gamma$: see Section \ref{sec-linearized sem} for the definition.
\begin{teor}\label{teor-sep cont tipo II}
Consider the family of parabolic problems with delay~\eqref{family}, $\w\in\W$
with $f$ satisfying conditions $\rm{(C1)}$-$\rm{(C4)}$, and assume that there exists a minimal set $K\subset
\W\times C_\gamma$ for the induced skew-product semiflow $\tau$. For the $n\times n$ real matrices
\begin{equation}\label{A and B}
 D_y f(\w,x,y,\wit y) = [a_{ij}(\w,x,y,\wit y)]\,, \quad
D_{\wit y} f(\w,x,y,\wit y) =  [b_{ij}(\w,x,y,\wit y)]
\end{equation}
define
\begin{align*}
\bar a_{ij} &= \sup\{ a_{ij}(\w,x,\varphi(0,x),\varphi(-1,x))\mid (\w,\varphi)\in K,\, x\in\bar U\}\,\;
\text{for }\, i\not= j\,, \; \text{and }\,\bar a_{ii}=0\,,
\\ \bar b_{ij} &= \sup\{ b_{ij}(\w,x,\varphi(0,x),\varphi(-1,x))\mid (\w,\varphi)\in K,\, x\in\bar U\}\,\;
\text{for }\, i\not= j\,, \; \text{and }\,\bar b_{ii}=0\,,
\end{align*}
and consider the matrix
\begin{equation}\label{A+B}
\bar A+ \bar B=[\bar a_{ij}+\bar b_{ij}]\,.
\end{equation}
Then, if the matrix $\bar A+ \bar B$ is
irreducible,
\begin{itemize}
\item[(i)]
there exists a $t_*\geq 1$ such that for each $(\w,\varphi)\in K$ the linear
operator $D_{\!\varphi} z_{t_*}(\w,\varphi)$ satisfies the following
dichotomy property: given $\psi\in C_\gamma$ with $\psi>0$,
either $D_{\!\varphi} z_{t_*}(\w,\varphi)\,\psi =0$ or
$D_{\!\varphi} z_{t_*}(\w,\varphi)\,\psi \gg 0$;
\item[(ii)] provided that $K$ admits a flow extension, there is a continuous separation of type~{\rm II} over $K$.
\end{itemize}
\end{teor}
\begin{proof}
This result is Theorem~5.1 in Novo et al.~\cite{nonuobsa} for the case of Neumann boundary conditions. The proof for Robin or Dirichlet boundary conditions follows step by step the same arguments, so that we only make some remarks.
\par
First of all, note that in the minimal set $K$ there are backwards extensions of semiorbits, and this implies that if $(\w,\varphi)\in K$, $\varphi\in C_\gamma$ has some specific regularity properties; more precisely $\varphi\in C^{1,2}([-1,0]\times \bar U,\R^n)$. This follows from Theorem~\ref{teor-space regularity-Neumann} or Theorem~\ref{teor-space regularity-Dirich}, moving backwards in the semiorbit with $t>2$ and then gaining regularity by coming back forwards.
\par
Second, when we look at the family of linearized IBV problems along the semiorbits of $(\w,\varphi)\in K$, the map $g$  in~\eqref{linear family} satisfies conditions (C), $(C^\theta(t))$ and $(C^{2\theta}(x))$ uniformly for $(\w,\varphi)\in K$, and Proposition~\ref{prop-monotone}~(ii) is repeatedly used.
\par
Finally, (ii) follows from the abstract Theorem~5.4 in Novo et al.~\cite{noos6} provided that the operators $D_{\!\varphi}z_t(\w,\varphi)$ are eventually compact, which happens for $t>1$. %This follows from Proposition 2.4 in Travis and Webb~\cite{trwe} in the Neumann or Robin cases, and in the Dirichlet case
\end{proof}
Before we state the main result, we give the appropriate definitions of uniform and strict persistence in the area above a compact $\tau$-invariant set $K\subset \Om\times C_\gamma$, which were introduced in Novo et al.~\cite{noos7} and in Obaya and Sanz~\cite{obsa}, respectively.
\begin{defi}\label{defi-persistence}
Let $K\subset \W\times C_\gamma$ be a compact
$\tau$-invariant set for the continuous and monotone semiflow $\tau$.

(i) The semiflow $\tau$ is said to be {\it
uniformly persistent} ({\it u-persistent} for short) in the region situated {\it strongly above}
$K$ if there exists a $\psi_0\in C_\gamma$, $\psi_0\gg 0$ such
that for any $(\w,\varphi)\in K$ and any $\phi\gg \varphi$  there
exists a time $t_0=t_0(\w,\varphi,\phi)$ such that $z_t(\w,\phi)\geq
z_t(\w,\varphi)+\psi_0$  for any $t\geq
t_0$.

(ii) The semiflow $\tau$ is said to be {\it strictly persistent at $0$} ({\it $s_0$-persistent} for short) in the region situated above
 $K$ if there  exists a collection of strictly positive maps $\psi_1,\ldots,\psi_N\in C_\gamma$, $\psi_i>0$ for every $i$, such that for any $(\w,\varphi)\in K$ and any $\phi\geq  \varphi$ with $\phi(0)> \varphi(0)$  there
exists a time $t_0=t_0(\w,\varphi,\phi)$ such that $z_t(\w,\phi)\geq
z_t(\w,\varphi)+\psi_i$ for any $t\geq
t_0$,  for one of the maps $\psi_1,\ldots,\psi_N$.
\end{defi}

\begin{teor}\label{teoremaDelay}
Consider the family of problems with delay~\eqref{family}, $\w\in\W$
with $f$ satisfying conditions $\rm{(C1)}$-$\rm{(C4)}$, and assume that there exists a minimal set $K\subset
\W\times C_\gamma$ for the induced skew-product semiflow $\tau$ which admits a flow extension.
 For each
$(\w,\varphi)\in  K$ consider the linearized IBV problem of~\eqref{family}
along the semiorbit of $(\w,\varphi)$, given in~\eqref{linear family}, and calculate the matrix
$\bar A+\bar B=[\bar a_{ij}+\bar b_{ij}]$ given in~\eqref{A+B}.

Without loss of generality, we can assume that the
matrix $\bar A+\bar B$  has the form
 \begin{equation}\label{triangular}
\left[\begin{array}{cccc}
\bar A_{11}+\bar B_{11}  & 0 &\ldots & 0 \\
\bar A_{21}+\bar B_{21}  & \bar A_{22} +\bar B_{22}&  \ldots& 0 \\
\vdots & \vdots &\ddots  & \vdots \\
\bar A_{k1}+\bar B_{k1} & \bar A_{k2}+\bar B_{k2} & \ldots& \bar A_{kk}+\bar B_{kk}
\end{array}\right]\,
\end{equation}
with irreducible diagonal blocks, denoted by  $\bar A_{11}+\bar B_{11},\ldots, \bar
A_{kk}+\bar B_{kk}$, of size $n_1,\ldots,n_k$ respectively
($n_1+\cdots + n_k=n$).
\par
For each $j=1,\ldots,k$, let us denote by $I_j$ the set formed by
the $n_j$ indexes corresponding to the rows of the block $\bar
A_{jj}+\bar B_{jj}$, and
 let $L_j$ be the linear skew-product semiflow induced on $K\times C([-1,0],\Pi_{i\in I_j}  E_i^\gamma)$
by the solutions of the $n_j$-dimensional linear systems for  $(\w,\varphi)\in K$ given by
\begin{equation}\label{bloque j}
%\hspace{-0,2cm}
\left\{\begin{array}{l} \des\frac{\partial u}{\partial t}=
 D_j\Delta u+A_{jj}(\w{\cdot}t,x,z(t,\w,\varphi)(x),z(t-1,\w,\varphi)(x))\,u(t,x)\\[.2cm]
  \; + B_{jj}(\w{\cdot}t,x,z(t,\w,\varphi)(x),z(t-1,\w,\varphi)(x))\,u(t-1,x)\,,\;\, t>0\,,\;x\in \bar U,
\\
\bar\alpha_j(x)\,u+\delta\,\des\frac{\partial u}{\partial n} =0\,,\quad  t>0\,,\;\,x\in \partial U,\\[.2cm]
u(s,x)=\psi^j(s,x)\,,\quad s\in [-1,0]\,,\;\,x\in \bar U,
\end{array}\right.
\end{equation}
for the corresponding diagonal blocks $A_{jj}$ and $B_{jj}$ of $D_y f$ and $D_{\wit y} f$ in~\eqref{A and B}, respectively,
 for $D_j$ and $\bar\alpha_j(x)$ respectively the $n_j\times n_j$-diagonal matrices with diagonal entries
$d_i$ and $\alpha_i(x)$ for $i\in I_j$, and initial value $\psi^j\in C([-1,0],\Pi_{i\in I_j}  E_i^\gamma)$. Then, $K^j=K\times \{0\}\subset K\times C([-1,0],\Pi_{i\in I_j}  E_i^\gamma)$ is a
minimal set for $L_j$ which admits a continuous separation of type~II. Let $\Sigma_p^j$ be its principal spectrum.
\par
If $k=1$, {\it i.e.}, if the matrix $\bar A+\bar B$ is irreducible, let $I=J=\{1\}$. Else, let
\begin{align*}
I&=\{j\in\{1,\ldots,k\} \,\mid\, \bar A_{ji}+\bar B_{ji}=0 \text{ for any } i\not= j\},\\
J&=\{j\in\{1,\ldots,k\} \,\mid\, \bar A_{ij}+\bar B_{ij}=0 \text{ for any } i\not= j\},
\end{align*}
that is, $I$ is composed by the indexes $j$ such that any block in the row of $\bar A_{jj}+\bar B_{jj}$, other than itself, is null, whereas $J$ contains those indexes $j$ such that any block in the column of $\bar A_{jj}+\bar B_{jj}$, other than itself, is null.
Then, some sufficient conditions for uniform and strict persistence at $0$ are the following:
\begin{itemize}
 \item[(i)] If $\Sigma_p^j\subset (0,\infty)$ for any $j\in I$, then  $\tau$ is uniformly persistent in the area situated strongly above $K$.
  \item[(ii)] If $\Sigma_p^j\subset (0,\infty)$ for any $j\in J$, then   $\tau$ is strictly persistent at $0$ in the area situated  above $K$.
 \end{itemize}
\end{teor}
\begin{proof}
We skip some details in the proof, since it often follows arguments in the proofs of Theorem~5.8 in Novo et al.~\cite{noos7} and Theorem~5.3 in Obaya and Sanz~\cite{obsa} for delay equations without diffusion, for (i) and (ii) respectively.
\par
Note that a convenient permutation of the variables takes the matrix $\bar A+\bar B$ into the form~\eqref{triangular}, and $\bar a_{ij},\,\bar b_{ij}\geq 0$ because of (C4) and the definition.
Also, we maintain the notation introduced in Theorem~\ref{teor-linearized sk} for the variational problems. Besides, for any map $v$, let us denote $v^j=(v_i)_{i\in I_j}$, for $j=1,\ldots,k$.
\par
To see (i), we distinguish three cases.
\par\noindent \textbf{(A1)}: $k=1$, that is, $\bar A+\bar B$ is an irreducible
matrix. Then Theorem~\ref{teor-sep cont tipo II} says that $K$ admits a continuous separation of type~II, and since $\Sigma_p^1\subset (0,\infty)$, the abstract Theorem~4.5 in~\cite{noos7} implies that $\tau$ is u-persistent in the area strongly above $K$.
\par\noindent \textbf{(A2)}: $k>1$ and  $\bar A+\bar B$ is a reducible matrix with a block diagonal structure.
In this case the argument goes exactly as in case (C2) in the proof of Theorem~5.8  in~\cite{noos7} for delay equations without diffusion. The key is to apply  Theorem~4.5 in~\cite{noos7} to each of the uncoupled linear skew-product semiflows $L_j$, which admit a continuous separation of type~II and have positive principal spectrums. In all, we find a map $\psi_0\gg 0$ and a  $t_0>0$ such that $D_{\!\varphi} z_{t}(\w,\varphi)\,\psi_0\gg 2\,\psi_0$ for $t\geq t_0$ and $(\w,\varphi)\in K$. Then,  Theorem~3.3 in~\cite{noos7} provides the u-persistence in the zone strongly above $K$.

\par\noindent \textbf{(A3)}: $k>1$ and  $\bar A+\bar B$ is a reducible matrix with a non-diagonal block lower triangular structure, that is, at least one of the non-diagonal blocks in \eqref{triangular} is not null. This time we combine the arguments in case (C3) in the proofs of Theorem~5.6 for PDEs and Theorem~5.8 for delay equations in~\cite{noos7}.   As in case (A2), the aim is to find a map $\psi\gg 0$ and a $t_1>0$ such that $D_{\!\varphi} z_{t}(\w,\varphi)\,\psi\gg 2\,\psi$ for $t\geq t_1$ and $(\w,\varphi)\in K$, so that Theorem~3.3 in~\cite{noos7} applies. Note that, since for $j\in I$ the systems \eqref{bloque j} are uncoupled, arguing as in (A2) we already have the appropriate maps $\psi_0^j\gg 0$  for $j\in I$ and the appropriate $t_0>0$, so that if $\psi\gg 0$ with $\psi^j=\psi_0^j$ for $j\in I$, then $v_t^j(\w,\varphi,\psi)\gg 2\,\psi^j$ for $t\geq t_0$ and $(\w,\varphi)\in K$, for each $j\in I$. That is, it remains to adequately complete the other components of $\psi\gg 0$.
\par
Since $1\in I$, we move forwards filling the gaps, so take $l=\min\{j\in \{2,\ldots,k\}\mid j\notin I\}\geq 2$. Then, at least one of the blocks to the left of $\bar A_{ll} +\bar B_{ll}$ is not null, that is, there exists an $m<l$, $m\in I$ such that $\bar A_{lm} +\bar B_{lm}\not=0$, so that  $\bar a_{i_1k}+\bar b_{i_1k}>0$ for some  $i_1\in I_l$ and $k\in I_m$.
For $u(t,x)=v(t,\w,\varphi,\psi)(x)\geq 0$ by Proposition \ref{prop-monotone}, from~\eqref{linear family}, the block lower triangular structure of the linearized systems, condition (C4), and since $k\in I_m$ with $m\in I$, we have that
\vspace{-0,2cm}
\begin{multline*}
\frac{\partial u_{i_1}}{\partial t}(t,x)=d_{i_1}\Delta u_{i_1}(t,x)+\sum_{j=1}^{i_1} \big( a_{i_1j}(\cdot)\,u_j(t,x)+b_{i_1j}(\cdot)\,u_j(t-1,x)\big)\\
\geq d_{i_1}\Delta u_{i_1}(t,x)+2\,a_{i_1k}(\cdot)\,(\psi_0^m)_k(0,x) +2\,b_{i_1k}(\cdot)\,(\psi_0^m)_k(-1,x) +a_{i_1i_1}(\cdot)\,u_{i_1}(t,x)
\end{multline*}
for $t\geq t_0$ and $x\in \bar U$, where $(\cdot)$ stands for $(\w{\cdot}t,x,z(t,\w,\varphi)(x),z(t-1,\w,\varphi)(x))$.
Then, we consider the auxiliar family of scalar parabolic PDEs for $(\w,\varphi)\in K$,
\begin{equation*}
\des\frac{\partial h}{\partial t}=
 d_{i_1}\Delta h+ 2\,a_{i_1k}(\cdot)\,(\psi_0^m)_k(0,x) +2\,b_{i_1k}(\cdot)\,\,(\psi_0^m)_k(-1,x)
 + a_{i_1i_1}(\cdot)\,h(t,x)
\end{equation*}
for $t>0$, $x\in \bar U$, with boundary condition $\alpha_{i_1}(x)\,h(t,x)+\delta\,\frac{\partial h}{\partial n}(t,x) =0$ for $t>0$ and $x\in \partial U$. Since $\bar a_{i_1k}+\bar b_{i_1k}>0$ means that  $a_{i_1k}(\w_1,x_1,\varphi_1(0,x_1),\varphi_1(-1,x_1))+b_{i_1k}(\w_1,x_1,\varphi_1(0,x_1),\varphi_1(-1,x_1))>0$
 for some $(\w_1,\varphi_1)\in K$ and $x_1\in U$, and $(\psi_0^m)_k(0,x_1), (\psi_0^m)_k(-1,x_1)>0$, one can apply the same dynamical argument used in Theorem~5.6 in~\cite{noos7} to conclude that there exist a $t_{i_1}>0$ and a map $\psi_{0i_1}\in E_{i_1}^\gamma$ with $\psi_{0i_1}\gg 0$ such that $h(t,\,\cdot\,,\w,\varphi,0) \gg 2 \,\psi_{0i_1}$ for any $(\w,\varphi)\in K$ and $t\geq t_{i_1}$. Note that a version of Lemma~2.11~(ii) in N\'{u}\~{n}ez et al.~\cite{nuos3} for Dirichlet boundary conditions in the intermediate space $E_{i_1}^\alpha$  has been used. Finally, consider $\psi_{0i_1}\in C([-1,0],E_{i_1}^\gamma)$ the identically equal to $\psi_{0i_1}$ map, which satisfies $\psi_{0i_1}\gg 0$, and take any initial condition $\psi\gg 0$ with $\psi^j=\psi_0^j$ for $j\in I$ and $\psi^l_{i_1}=\psi_{0i_1}$.  Then, comparing solutions of the two previous problems (see Martin and Smith~\cite{masm0,masm}), we can conclude that $(v_{i_1}^l)_t(\w,\varphi,\psi)\gg 2\,\psi_{i_1}^l$ for $t\geq t_0+t_{i_1}+1$ and $(\w,\varphi)\in K$, and we are done with the component $i_1\in I_l$.
\par
The argument for the rest of components in $I_l$, if any, is similar and relies on the irreducibility of the block $\bar A_{ll}+\bar B_{ll}$; and for the remaining blocks, if any, is just the same. The proof of (i) is finished.
\par
To see (ii) we consider again three cases, in accordance with Theorem~5.3 in~\cite{obsa}.
\par\noindent \textbf{(B1)}: $k=1$, that is, $\bar A+\bar B$ is an irreducible
matrix. By (i), we already know that $\tau$ is u-persistent. To see that it is also $s_0$-persistent, take $\psi_0\gg 0$ the map given in Definition~\ref{defi-persistence}~(i) and  $t_*\geq 1$ the time given in Theorem~\ref{teor-sep cont tipo II}~(i). Now take  $(\w,\varphi)\in K$ and $\phi\geq \varphi$ with $\phi(0)> \varphi(0)$. Then, $\phi_i(0)> \varphi_i(0)$ for some $i$ and Proposition~\ref{prop-monotone}~(ii) applied to the linearized systems implies that $v_i(t,\w,\varphi,\phi-\varphi)\gg 0$ for any $t>0$. Then it cannot be $D_{\!\varphi} z_{t_*}(\w,\varphi)\,(\phi-\varphi) =0$, and necessarily
$D_{\!\varphi} z_{t_*}(\w,\varphi)\,(\phi-\varphi) \gg 0$. By continuity,  $D_{\!\varphi} z_{t_*}(\w,\lambda\phi+(1-\lambda)\varphi)\,(\phi-\varphi) \gg 0$ for $\lambda\in[0,\varepsilon]$ for a certain $\varepsilon>0$, and using~\eqref{mean value}, $z_{t_*}(\w,\phi)\gg z_{t_*}(\w,\varphi)$. To finish, apply the u-persistence to $(\w{\cdot}t_*,z_{t_*}(\w,\varphi))\in K$ together with the semicocycle property~\eqref{semicocycle}.
\par
Now, for $k>1$, take $\phi\geq \varphi$ with $\phi(0)> \varphi(0)$ and distinguish two possibilities:

\par\noindent \textbf{(B2)}: $k> 1$ and   $\phi_i(0)> \varphi_i(0)$ for some $i\in I_j$ with $j\in J$. In this case we follow the arguments in case (C2) in the proof of Theorem~5.3 in~\cite{obsa} for delay equations without diffusion. Basically, a family of $n_j$-dimensional systems of nonlinear parabolic PFDEs with delay over the base flow in $K$ is built, in such a way that it is a minorant family for the components  $y^j(t,x)=z^j(t,\w,\varphi)(x)$, and besides the linearized systems along the orbits in a minimal set are precisely the systems~\eqref{bloque j}, with irreducible matrix $\bar A_{jj}+\bar B_{jj}$. Then, to this family case (B1) applies, and thus there exist a $\psi_0^j\in C([-1,0],\Pi_{i\in I_j}  E_i^\gamma)$, $\psi_0^j\gg 0$ and a $t_0^j>0$, associated to its u-persistence. Then, using standard arguments of comparison of solutions, one can check  that $z_t(\w,\phi)\geq
z_t(\w,\varphi)+\psi_j$ for any $t\geq
t_0^j$, for the map $\psi_j \in C_\gamma$ defined by $\psi_j^j=\psi_0^j$ and $\psi_j^m=0$ if $m\not=j$, which satisfies $\psi_j>0$. Just remark that the maps $\{\psi_j\}_{j\in J}$ built in this way are the appropriate collection required in Definition~\ref{defi-persistence}~(ii).

\par\noindent \textbf{(B3)}: $k> 1$ and  $\phi^l(0)=\varphi^l(0)$ for any $l\in J$. Then, consider $i$ such that $\phi_i(0)> \varphi_i(0)$ with $i\in I_j$ for some $j\notin J$. Now we distinguish two situations:

\par\noindent \textbf{(B3.1)}: There exists an $m\geq 1$ such that $\bar A_{j+m,j}+\bar B_{j+m,j}\not=0$ with $j+m\in J$. In this case we search for a time $t_1>0$ such that $z^{j+m}(t_1,\w,\phi)>z^{j+m}(t_1,\w,\varphi)$, for then we can apply case (B2) together with the semicocycle relation~\eqref{semicocycle}.
\par
As a first step, let us study the components $v_t^j(\w,\varphi,\phi-\varphi)$. Write $L_j(t,\w,\varphi,\psi^j)=(\tau(t,\w,\varphi),w_t(\w,\varphi,\psi^j))$ for the linear skew-product semiflow induced by the solutions of \eqref{bloque j} for $(\w,\varphi)\in K$ and $\psi^j\in C([-1,0],\Pi_{i\in I_j}  E_i^\gamma)$. By condition (C4), a comparison of solutions argument says that $v_t^j(\w,\varphi,\phi-\varphi)\geq w_t(\w,\varphi,\phi^j-\varphi^j)$ for $t\geq 0$. Besides, applying Proposition~\ref{prop-monotone}~(ii) to $L_j$, since $\phi^j_i(0)>\varphi^j_i(0)$, we get that $w_i(t,\w,\varphi,\phi^j-\varphi^j)\gg 0$ for any $t>0$. Therefore, it must be $w_{t_*}(\w,\varphi,\phi^j-\varphi^j)\gg 0$ for $t_*\geq 1$ the time given in Theorem~\ref{teor-sep cont tipo II}~(ii) for $L_j$. Then, the linear semicocycle property~\eqref{linear semicocycle} and Proposition~\ref{prop-monotone}~(ii) imply that $w_t(\w,\varphi,\phi^j-\varphi^j)\gg 0$ for $t\geq t_*$, so that also $v_{t}^j(\w,\varphi,\phi-\varphi)\gg 0$ for $t\geq t_*$.
\par
Finally, take $i_1\in I_{j+m}$ and $k\in I_j$ such that $\bar a_{i_1k}+\bar b_{i_1k}>0$. Then, for $u(t,x)=v(t,\w,\varphi,\phi-\varphi)(x)$ recall that with (C4), $u(t,x)\geq 0$ by Proposition~\ref{prop-monotone}~(i), and since $k\in I_j$, $u_{k}(t,x)> 0$ for any $t\geq t_*-1$ and $x\in U$. Now, arguing as in the proof of Theorem~5.1 in Novo et al.~\cite{nonuobsa}, associated to the minimal set $K$ and to the open set $U_{i_1k}=\{(\tilde \w,\tilde \varphi)\in \Om\times C_\gamma\mid a_{i_1k}(\tilde\w,x,\tilde\varphi(0,x),\tilde\varphi(-1,x))+b_{i_1k} (\tilde\w,x,\tilde\varphi(0,x),\tilde\varphi(-1,x))>0 \;\text{for some } x\in U\}$,    there exists a $T_0>2$ such that for any $(\tilde \w,\tilde \varphi)\in K$ there is a $t_0\in (2,T_0)$ such that $\tau(t_0,\tilde\w,\tilde \varphi)\in U_{i_1k}$. Applying this property to $\tau(t_*,\w,\varphi)\in K$ there exist a $t_0\in (2,T_0)$ and an $x_0\in U$ such that
\begin{multline*}
\wit a_{i_1k} + \wit b_{i_1k}:= a_{i_1k}(\w{\cdot}(t_*+t_0),x_0,z(t_*+t_0,\w,\varphi)(x_0),z(t_*+t_0-1,\w,\varphi)(x_0))\\
+b_{i_1k}(\w{\cdot}(t_*+t_0),x_0,z(t_*+t_0,\w,\varphi)(x_0),z(t_*+t_0-1,\w,\varphi)(x_0))>0\,.
\end{multline*}
Now, by (C4), on the one hand $u_{i_1}(t,x)$ satisfies the following parabolic inequality
\[
\frac{\partial u_{i_1}}{\partial t}(t,x)
\geq d_{i_1}\Delta u_{i_1}(t,x) +a_{i_1i_1}(\w{\cdot}t,x,z(t,\w,\varphi)(x),z(t-1,\w,\varphi)(x))\,u_{i_1}(t,x)
\]
for $t> t_*$ and $x\in \bar U$,  together with the corresponding boundary condition. Then, if it were $u_{i_1}(t_*+t_0,x_0)=0$, the minimum
principle for scalar parabolic PDEs would say that $u_{i_1}(t,x)=0$ for any $(t,x)\in
[t_*,t_*+t_0]\times \bar U$, so that in particular $\Delta  u_{i_1}(t_*+t_0,x_0)=0$
and $\partial_t u_{i_1}(t_*+t_0,x_0)=0$. But on the other hand, then
\[
\frac{\partial u_{i_1}}{\partial t}(t_*+t_0,x_0) \geq
\wit a_{i_1k}\,\,u_k(t_*+t_0,x_0)
+ \wit b_{i_1k}\,\,u_k(t_*+t_0-1,x_0)>0\,,
\]
a contradiction. Therefore, $u_{i_1}(t_*+t_0,x_0)>0$, so that $v_{i_1}(t_*+t_0,\w,\varphi,\phi-\varphi)>0$ and by Proposition~\ref{prop-monotone}~(ii), $v_{i_1}(t,\w,\varphi,\phi-\varphi)\gg 0$ for any $t>t_*+t_0$. Take such a $t_1>t_*+t_0$ and use relation \eqref{mean value componente} together with a continuity  argument to conclude that $z_{i_1}(t_1,\w,\phi)\gg z_{i_1}(t_1,\w,\varphi)$, so that $z^{j+m}(t_1,\w,\phi)>z^{j+m}(t_1,\w,\varphi)$, as we wanted.
\par\noindent \textbf{(B3.2)}: For any $m\geq 1$ such that $\bar A_{j+m,j}+\bar B_{j+m,j}\not=0$,  $j+m\notin J$. In this case, we take the greatest $m\geq 1$ such that $\bar A_{j+m,j}+\bar B_{j+m,j}\not=0$ and we argue as in case (B3.1) to find a $t_1>0$ such that $z^{j+m}(t_1,\w,\phi)>z^{j+m}(t_1,\w,\varphi)$. Since $j+m\notin J$, again there is an $l\geq 1$ such that $\bar A_{j+m+l,j+m}+\bar B_{j+m+l,j+m}\not=0$. If for some such $l\geq 1$, $j+m+l\in J$ we fall again in case (B3.1), and if not, we are again in case (B3.2) and we just iterate the procedure. Since $k\in J$, in a finite number of iterations we fall in case (B3.1). The proof is finished.
\end{proof}
%%%%%%%%%%%%%%%%%%%%%%%%%%%%%%%%%%%%%%%%%%%%%%%%%%%%%%%%%%%%%%%%%%%%%%%%%%%
%%%%%%%%%%%%%%%%%%%%%%%%%%%%%%%%%%%%%%%%%%%%%%%%%%%%%%%%%%%%%%%%%%%%%%%%%%%%%%%%%%%%%%%%

\end{document}